\documentclass[10pt,a4paper,oneside,reqno,english]{smfart}

\usepackage{amssymb}
\usepackage{mathtools}
\usepackage{tensor,slashed}
\usepackage{physics}
\usepackage{smfthm}
\usepackage[all,cmtip]{xy}
\usepackage[colorlinks=true,linkcolor=blue,urlcolor=blue,citecolor=blue,anchorcolor=blue]{hyperref}
\usepackage[shortlabels]{enumitem}

\usepackage[T1]{fontenc}

\usepackage{lmodern}
\usepackage[utf8]{inputenc}

\usepackage[english]{babel}

\usepackage[dvipsnames]{xcolor}

\usepackage{tikz}
\usetikzlibrary{cd}
\usetikzlibrary{decorations.markings}
\tikzset{double line with arrow/.style args={#1,#2}{decorate,decoration={markings,%
mark=at position 0 with {\coordinate (ta-base-1) at (0,1pt);
\coordinate (ta-base-2) at (0,-1pt);},
mark=at position 1 with {\draw[#1] (ta-base-1) -- (0,1pt);
\draw[#2] (ta-base-2) -- (0,-1pt);
}}}}

\usepackage{microtype}
\usepackage{hyperref}

\DeclareRobustCommand{\SkipTocEntry}[5]{}

\newcommand{\numberset}{\mathbb}

\newcommand{\N}{\numberset{N}} 
\newcommand{\Z}{\numberset{Z}} 
 
\newcommand{\R}{\numberset{R}}
\let\C\relax
\newcommand{\C}{\numberset{C}}
\newcommand{\T}{\numberset{T}}

\newcommand{\cK}{\mathcal{K}}
\newcommand{\G}{\mathcal{G}}
\newcommand{\cH}{\mathcal{H}}

\newcommand{\bu}{\mathbf{u}}
\newcommand{\bt}{\mathbf{t}}
\newcommand{\bn}{\mathbf{n}}

\DeclareMathOperator{\Chern}{Ch}

\DeclareMathOperator{\Aut}{Aut}

\DeclareMathOperator{\KKK}{KK}

\newcommand{\hatotimes}{\widehat{\otimes}}
\newcommand{\hot}{\,\hatotimes\,}

\DeclareMathOperator{\Ind}{Ind}

\theoremstyle{plain}
\newtheorem{theoremA}{Theorem}

\newtheorem{corollaryA}[theoremA]{Corollary}

\newtheorem*{question*}{Question}
\newtheorem*{convention*}{Convention}
\newtheorem*{assumption*}{Assumption}
\newtheorem*{assumptions*}{Assumptions}

\author{Hao Guo}
\address{Yau Mathematical Sciences Center, Tsinghua University}
\email{haoguo@mail.tsinghua.edu.cn}
\author{Valerio Proietti}
\address{Graduate School of Mathematical Sciences, The University of Tokyo, Japan}
\curraddr{Department of Mathematics, University of Oslo, P.O. box 1053, Blindern, 0316 Oslo, Norway}
\email{valeriop@math.uio.no}
\author{Hang Wang}
\address{School of Mathematical Sciences, East China Normal University}
\email{wanghang@math.ecnu.edu.cn}
\title[A geometric Elliott invariant and noncommutative rigidity]{A geometric Elliott invariant and noncommutative rigidity of mapping tori}

\begin{document}
\frontmatter

\begin{abstract}
We prove a rigidity property for mapping tori associated to minimal topological dynamical systems using tools from noncommutative geometry. More precisely, we show that under mild geometric assumptions, a leafwise homotopy equivalence of two mapping tori associated to $\Z^d$-actions on a compact space can be lifted to an isomorphism of their foliation $C^*$-algebras. This property is a noncommutative analogue of topological rigidity in the context of foliated spaces whose space of leaves is singular, where isomorphism type of the $C^*$-algebra replaces homeomorphism type. Our technique is to develop a geometric approach to the Elliott invariant that relies on topological and index-theoretic data from the mapping torus. We also discuss how our construction can be extended to slightly more general homotopy quotients arising from actions of discrete cocompact subgroups of simply connected solvable Lie groups, as well as how the theory can be applied to the magnetic gap-labelling problem for certain Cantor minimal systems.
\end{abstract}

\maketitle
\setcounter{tocdepth}{2}
\vspace*{-1ex}     
\tableofcontents

\mainmatter

\addtocontents{toc}{\SkipTocEntry}\section*{Introduction}

It is a fundamental question in topology whether a given homotopy equivalence can be lifted to a homeomorphism (e.g. the Borel conjecture). Such rigidity questions are notoriously difficult, and in certain geometric situations it is more appropriate to replace the notion of homeomorphism by isomorphism in a different category. In this paper, we study a noncommutative-geometric analogue of this problem for a class of foliated spaces, namely mapping tori associated to certain minimal topological dynamical systems. The minimality of such systems means that the leaf space of the foliation is singular, and a natural approach from noncommutative geometry is to work with the isomorphism type of the foliation $C^*$-algebra.

Before formulating the problem, let us provide some background on the systems we will work with. Let $X$ be a compact metrizable space and $G$ a discrete group. Let
$$\alpha\colon G\to \mathrm{Homeo}(X)$$ 
be an action that is free
and minimal (meaning each $G$-orbit is
dense in $X$). We refer to the data $(X,G,\alpha)$ as a \emph{minimal topological dynamical system}. Minimality implies that the orbit space $X/G$ is a pathological space whose topology consists of exactly two open sets. Yet systems of this type often appear in physical settings, so it is important to find a framework in which it is possible to calculate analogues of topological invariants and formulate the appropriate versions of classical topological questions. One replacement for the space $X/G$ provided by algebraic topology is the homotopy quotient (or the Borel construction) $Y_\alpha=X\times_{G} EG$. On the other hand, by virtue of the Gelfand-Naimark duality, noncommutative geometry provides another possibility, namely the (reduced) crossed product $C^*$-algebra 
$$A_\alpha=C(X)\rtimes_{\alpha} G.$$

Motivated by this, we formulate and prove a version of the rigidity question in the case $G=\Z^d$ and $EG=\R^d$. In this case $Y_\alpha$, known as the \emph{mapping torus} for the $\Z^d$-action, is a foliated space whose space of leaves can be identified with $X/G$. The algebra $A_\alpha$ is Morita equivalent to the $C^*$-algebra of the foliation. We prove the following (see Theorem \ref{thm:lift} in the text for details):

\begin{theoremA}
\label{thm main1}
Let $X$ be a compact Hausdorff space with finite covering dimension endowed with two free and minimal $\Z^d$-actions $\alpha$ and $\beta$. Denote by $Y_\alpha$ and $Y_\beta$ the mapping tori associated to $(X,\alpha)$ and $(X,\beta)$, respectively. Then any basic leafwise homotopy equivalence $Y_\alpha\cong Y_\beta$ can be lifted to an isomorphism of $C^*$-algebras $A_\alpha\cong A_\beta$.
\end{theoremA}

Our approach to Theorem \ref{thm main1} uses $K$-theory and foliation index theory, together with results from the classification of $C^*$-algebras, to construct a geometric analogue of the Elliott invariant for the class of $C^*$-algebras we work with. Recall that the Elliott invariant is a fundamental construction in classification theory that controls the isomorphism types of many nuclear $C^*$-algebras, including the algebras $A_\alpha$ and $A_\beta$ from Theorem \ref{thm main1} (see Theorem \ref{thm:classification} and the references given there). 

In general, for $C^*$-algebras that arise from an underlying geometric or dynamical system, it is conceivable that the Elliott invariant can be recovered from geometric data. In pursuing this goal, we are led to build a dictionary of sorts between geometric properties of the mapping torus and $C^*$-algebraic properties of the associated crossed product. We hope in the process to shed some light on the interaction between noncommutative geometry and operator algebras in the context of classification.

We use the foliation index theorem to establish a relationship between trace values that form part of the Elliott invariant and topological data associated to the foliation, such as characteristic numbers obtained by integrating tangential cohomology classes. For the $C^*$-algebras under consideration, we are able to construct a geometric invariant that is analogous to the Elliott invariant, as summarised by the following (Theorem \ref{thm:classres} in the text):

\begin{theoremA}
\label{thm main2}
The crossed products $C(X)\rtimes_{\alpha}\Z^d$, possibly twisted by a $\Z^d$-cocycle, are classified by the following data, taken together:
\begin{itemize}[$\bullet$]
\item The topological $K$-theory of the mapping torus $Y_\alpha$ associated to $(X,\Z^d,\alpha)$;
\item The $K$-theory map induced by ``inclusion of the fundamental domain'' inside $Y_\alpha$;
\item The structure of the top tangential de Rham homology group of $Y_\alpha$;
\item The index pairing of the leafwise Dirac operator twisted by bundles over $Y_\alpha$.
\end{itemize}
\end{theoremA}
The terminology used here is explained in Section \ref{sec:two} and Section \ref{sec:index}.

The techniques that underlie Theorem \ref{thm main2} can be used to establish a number of other related results. For example, we prove the following (Theorem \ref{thm:class} in the text):

\begin{theoremA}
\label{thm A}
Suppose the Chern character $K^*(Y_\alpha)\to H^*_\tau(Y_\alpha,\R)$ with values in tangential cohomology is surjective after tensorization with $\R$. Then the crossed products $C(X)\rtimes_{\alpha}\Z^d$ are classified by ordered $K$-theory.
\end{theoremA}
(We note that, unlike the ordinary Chern character, the tangential Chern character is not in general a rational isomorphism, even if the foliated space is a manifold.) We will also discuss a generalisation of Theorem \ref{thm main2} to the setting where $\Z^d$ is replaced by a discrete cocompact subgroup of a simply connected solvable Lie group, which includes the case of the discrete and continuous Heisenberg groups.

Theorem \ref{thm main2} can also be applied to clarify existing results in the literature concerning properties of the crossed product algebra under constraints on $Y_\alpha$. We give an example of such an application by combining Theorem \ref{thm main2} with the main result in \cite{benmat:mag} to obtain the following result:

\begin{corollaryA}
\label{cor:gap equivalence}
The $C^*$-algebra of observables associated to the (magnetic) gap-labelling problem for principal solenoidal tori is isomorphic to a crossed product $C^*$-algebra arising from a (twisted) $\Z^d$-odometer system.
\end{corollaryA}

(The statement above is slightly ambiguous for the sake of brevity; see Subsection \ref{subsec:app} for more details.)
The gap-labelling conjecture stems from a noncommutative approach to Bloch theory for aperiodic materials \cite{bel:gap}. Corollary \ref{cor:gap equivalence} helps to clarify the type of crossed product algebras involved in the strategy devised in \cite{benmat:mag}, which assumes a favourable topological structure for the mapping torus. By applying our geometric construction of the Elliott invariant, we are able to translate the relevant topological constraints to operator-algebraic constraints, and thereby identify the corresponding crossed products as those arising from $\Z^d$-odometers. Nonetheless, it is worth pointing out that the original gap-labelling conjecture was formulated for $\Z^d$-minimal systems arising from tiling spaces. These systems are in general more complicated than odometers \cite{putand:til,valmak:groupoid,valmak:groupoidtwo}.

The paper is set out as follows. We begin Section \ref{sec:prelims} by collecting some relevant background that will form the basis for our subsequent discussion. We then set up the problem of identifying the Elliott invariant using properties of the mapping torus. The details of this identification are the subject of the next two sections, which form the technical core of the paper. In Section \ref{sec:two}, we investigate  topological conditions under which the order unit in the Elliott invariant is preserved by the naturally induced map on crossed products. Section \ref{sec:index} concerns the identification of trace pairings. In Section \ref{sec:results} we prove the main results of the paper, Theorems \ref{thm main1} and \ref{thm main2}. We also discuss further applications of the theory, including to the gap-labelling conjecture for principal solenoidal tori and $\Z^d$-odometer systems, and discuss the generalisation to discrete cocompact subgroups of simply connected Lie groups.

\addtocontents{toc}{\SkipTocEntry}\section*{Acknowledgements}

We would like to thank J.~Gabe and Z.~Niu for many helpful suggestions concerning the classification theory of $C^*$-algebras. The
second named author would like to thank P.~Piazza for suggesting the gap-labelling conjecture and foliations as research topics. 

\section{Preliminaries}
\label{sec:prelims}
\subsection{Basic setup}
\label{subsec:basic}
Suppose $\alpha$ is a free and minimal left action of $G=\Z^d$ on the compact metrizable space $X$. Let $Y_\alpha=X\times_{\Z^d}\R^d$ denote the quotient by the equivalence relation
$$(x,\mathbf{u})\sim(\alpha(-\mathbf{n})(x),\mathbf{u}+\mathbf{n}).$$
The space $Y_\alpha$ is commonly known as the \emph{mapping torus} for the action $\alpha$, being a
fibre bundle
$$
\begin{tikzcd}
X\arrow{r} & Y_\alpha\arrow{d}\\
& \T^d
\end{tikzcd}
$$
over the standard torus. 
\begin{rema}
In the case that $d=1$, the action of $\Z$ on $X$ is generated by a single homeomorphism $\alpha\colon X\to X$, so the situation reduces to the classical mapping torus from topology. It is useful to keep in mind the case of $X=S^1$, with $\Z$ acting by a rotation; in that case, the mapping torus is topologically the $2$-torus and can also be thought of as a cylinder $S^1\times[0,1]$, with ends identified by a twist.
\end{rema}

The mapping torus can also be thought of as a suspension of the action $\alpha$. Indeed, there is a flow on $Y_\alpha$ defined by the left $\R^d$-action
\begin{equation}
\label{eq Rd action}	
\tilde{\alpha}(\mathbf{s})[x,\mathbf{u}] =[x,\mathbf{u}+\mathbf{s}].
\end{equation}
From the data $(X,\alpha)$, we may construct the transformation groupoids 
\begin{align}
\label{eq groupoids}
\G_\alpha&=\Z^d\ltimes_\alpha X,\nonumber\\
\tilde{\G}_\alpha&=\R^d\ltimes_{\tilde\alpha}Y_\alpha.
\end{align} 
Note that $\G_\alpha$ is isomorphic to the subgroupoid
$\Z^d\ltimes_{\tilde{\alpha}}(X\times_{\Z^d}\Z^d)\subseteq\tilde{\G}_\alpha$ via the map
\begin{align}
\label{eq subgroupoid}
\left(x,\mathbf{n},\alpha(\mathbf{n})(x)\right)\mapsto\left([x,0],\mathbf{n},[x,\mathbf{n}]\right).
\end{align}
We set $A_\alpha=C(X)\rtimes_{\alpha}\Z^d$ for brevity, and we have identifications
\begin{align*}
C^*(\G_\alpha)&=C(X)\rtimes_{\alpha}\Z^d=A_\alpha,\\
C^*(\tilde{\G}_\alpha)&=C(Y_\alpha)\rtimes_{\tilde{\alpha}}\R^d.
\end{align*}
Here, since $\Z^d$ and $\R^d$ are amenable, we may equivalently take the full or reduced crossed products.

\begin{rema}
In the above, we would gain no further generality by assuming the action to be only \emph{topologically} free -- meaning the fixed-point set of any non-identity element has empty interior -- since the fixed points for any one group element in an abelian group form a closed invariant set.
\end{rema}
\begin{convention*}
When it is clear from context, we may omit the subscript $\alpha$ in $A_\alpha$, $Y_\alpha$, or $\G_\alpha$, as well as new objects defined from them.	
\end{convention*}
We will often use the following two well-known results (see \cite[Examples 2.6 \& 2.7]{murewi:morita} and \cite[Section 3]{put:spiel}).

\begin{lemm}\label{lem:morita}
Let $s\colon \tilde{\G}\to Y$ denote the source map. Then the subspace $s^{-1}(X\times\{0\})$ is a $\G$-$\tilde{\G}$ principal bibundle. In particular, $C^*(\G)$ is strongly Morita equivalent to $C^*(\tilde{\G})$.
\end{lemm}

The celebrated Connes' Thom isomorphism \cite{connes:thom} states that $K_{*+d}(B)\cong K_*(B\rtimes \R^d)$ for any $C^*$-algebra
$B$ with an $\R^d$-action. In our setting this implies that $K^{*+d}(Y)\cong K_*(C^*(\tilde{\G}))$. Then Lemma \ref{lem:morita} and Morita invariance of $K$-theory imply the following.

\begin{prop}\label{prop:tcob}
There is an isomorphism $K^{*+d}(Y) \cong K_*(A)$.
\end{prop}

Let us put the groupoid $\tilde{\G}$ in a broader context, in the language of foliation theory
	(our conventions are based on \cite{mooscho:globanalysis}). The space $Y$ is foliated by the orbits of the $\R^d$-action \eqref{eq Rd action}. The tangent bundles of the individual leaves together form the \emph{foliation bundle} $FY$, a real vector bundle of rank $d$ that is orientable in this case. One possible choice of transversal for this foliated space is 
	\begin{equation}
	\label{eq N}
	N=[X\times\{0\}]\coloneqq p(X\times \{0\})\subseteq Y,
	\end{equation}
	where $p\colon X\times\R^d\to Y$ is the quotient map. In general, the \emph{foliation $C^*$-algebra}
	is defined to be the $C^*$-algebra associated to the holonomy groupoid of the foliated space,
	a quotient of the monodromy
	groupoid, i.e. the groupoid of homotopy classes of paths lying entirely on one leaf (the relation between holonomy and monodromy groupoids is explored in detail in  \cite{phil:hol}). However, the situation simplifies in our setting: since the leaves of $Y$ are contractible, these two groupoids coincide and are in fact identical to the transformation groupoid $\tilde{\G}$. The foliation $C^*$-algebra is then the algebra $C^*(\tilde\G)$ considered above.
	

\subsection{Twisted crossed products}
\label{subsec twisted}
We briefly digress to discuss a generalisation of the above to the case of twisted crossed products of $C(X)$ by $\Z^d$ (cf. the setting in \cite{benmat:magzero}). Given a matrix $\Theta\in M_d(\R)$, define a map $\sigma\colon\Z^d\times\Z^d\to U(1)$ by 
\[
\sigma(\mathbf{m},\mathbf{n}) = e^{-\pi i\, \mathbf{m}^t\Theta\mathbf{n}}.
\]
Then $\sigma$ is a normalised group $2$-cocycle for $\Z^d$, and one can use it to form a twisted crossed product $C^*$-algebra, which we will denote by
\begin{equation}
\label{eq A theta}
A^\Theta=C(X)\rtimes_\sigma\Z^d=C^*(\G,\Theta).
\end{equation}
(Again, we have suppressed here the action $\alpha$ on $X$.)
Note that $\sigma$ is homotopic to the trivial $2$-cocycle via the path $s\mapsto\sigma^s$, for $s\in[0,1]$, where $\sigma^s(\mathbf{m},\mathbf{n})=e^{-\pi i s \, \mathbf{m}^t\Theta\mathbf{n}}$. It follows from \cite[Theorem 4.2]{parae:twisttwo} 
that we have an isomorphism
\begin{equation}
\label{eq twisted untwisted}
\phi\colon K_*(A^\Theta)\cong K_*(A).
\end{equation}
(See also the
proof of Lemma \ref{lem:twistunit} in the next section.) Let us also introduce the groupoid 
$$\bar{\G}=(X\times\R^d\times\R^d)/\Z^d,$$ 
following \cite{benmat:magzero}. Here the action of $\Z^d$ on $X\times\R^d\times\R^d$ is diagonal, i.e.
$$\mathbf{n}\cdot(x,\mathbf{v},\mathbf{u})=(\alpha(-\mathbf{n})(x),\mathbf{v}+\mathbf{n},\mathbf{u}+\mathbf{n}),$$ 
with range $r[x,\mathbf{v},\mathbf{u}]=[x,\mathbf{v}]$, source $s[x,\mathbf{v},\mathbf{u}]=[x,\mathbf{u}]$, and multiplication
\[
[x,\mathbf{w},\mathbf{v}]\circ [x,\mathbf{v},\mathbf{u}]= [x,\mathbf{w},\mathbf{u}].
\]
One verifies that $\bar{\G}$ is isomorphic to the groupoid $\tilde{\G}$ from \eqref{eq groupoids} via the map 
\begin{align}
\label{eq iso Gtilde Gbar}
\tilde{\G}\to \bar{\G},\quad ([x,\bt],\bu,[x,\bt+\bu])\mapsto[x,\bt+\bu,\bt].
\end{align}
Note that this isomorphism endows $\bar{G}$ with a (right) Haar system $\{\nu_y\}_{y\in Y}$, obtained by transporting the standard system on $\tilde{G}$ induced by the Lebesgue measure on $\R^d$.

We now define a $C^*$-algebra $C^*(\tilde{\G},\Theta)$, by using the $\bar\G$ picture of $\tilde\G$, whose basic elements are continuous kernels that satisfy a certain covariance condition defined by $\Theta$. To formulate this, consider the closed $2$-form 
$$B=\frac{1}{2}d\mathbf{u}^t\Theta d\mathbf{u}$$ 
on $X\times \R^d$ (independent of $x\in X$). Write $\mathbf{u}=(u_1,\ldots,u_d)$, and let 
$$\eta=\sum_{j<k}\Theta_{jk}u_jdu_k.$$ 
Observe that $B=d\eta$. Let $m^*$ denote the
pullback map associated to the diffeomorphism of translation by $\mathbf{m}\in \Z^d$, so that $d(\eta-m^*\eta)=0$. Since there are no obstructions to exactness on $\R^d$, we can find a smooth function 
\begin{align}
\label{varphi}
	\varphi_{\mathbf{m}}\colon\R\to\R 
\end{align}
satisfying $d\varphi_{\mathbf{m}}=\eta-m^*\eta$ and $\varphi_{\mathbf{m}}(0)=0$.

The elements $k\in C_c(X\times\R^d\times\R^d)$ form a $\ast$-algebra under convolution and adjoint defined as usual using composition and inversion in $\bar{G}$. This algebra can be normed as follows:
\[
\lVert k \rVert_1=\sup_{y\in Y}\Bigl\{\max \biggl( \int \lvert k(x,v,u)\rvert \,d\nu_y,\,\int \lvert k(x,v,u)\rvert \,d\nu^y \biggr)\Bigr\},
\]
where the Haar system $\{\nu_y\}_{y\in Y}$ is obtained from Eq. \eqref{eq iso Gtilde Gbar}.
\begin{defi}
\label{def c g tilde theta}
Let $C^*(\tilde{\G},\Theta)$ be the enveloping $C^*$-algebra of the $\lVert \cdot \rVert_1$-completion of the subalgebra of continuous functions $k\in C_c(X\times\R^d\times\R^d)$ satisfying
\begin{equation}\label{eq:covariance}
e^{-\pi i \varphi_{\mathbf{n}}(\mathbf{v})}k(\alpha(-\mathbf{n})(x),\mathbf{v}+\mathbf{n},\mathbf{u}+\mathbf{n})e^{\pi i\varphi_{\mathbf{n}}(\mathbf{u})}=k(x,\mathbf{v},\mathbf{u}),
\end{equation}
where $\bn\in\Z^d$ and $\varphi_{\mathbf{n}}$ is defined in Equation \eqref{varphi}.
\end{defi}

Consider the subgroupoid $\cH=[X\times\Z^d\times\Z^d]$ of $\bar{\G}$. The map \eqref{eq iso Gtilde Gbar} restricts to an isomorphism $\Z^d\ltimes(X\times_{\Z^d}\Z^d)\cong\cH$. Composing this with the isomorphism \eqref{eq subgroupoid} gives an isomorphism $\G\cong\cH$, which we can use to endow $\cH$ with a right Haar system coming from the counting measure on $\Z^d$. 

Define $C^*(\cH,\Theta)$ as the enveloping $C^*$-algebra in analogy with Definition \ref{def c g tilde theta}, from the functions $k\in C_c(X\times\Z^d\times\Z^d)$ satisfying the same covariance condition \eqref{eq:covariance}.

\begin{prop}
\label{prop kernel iso}
The twisted crossed product $C^*$-algebra $A^\Theta$ defined in \eqref{eq A theta} is $*$-isomorphic to $C^*(\cH,\Theta)$.
\end{prop}
\begin{proof}
The groupoid isomorphism $\G\cong\cH$ given above (obtained by restricting \eqref{eq iso Gtilde Gbar} and composing with \eqref{eq subgroupoid}) sends $(x,\bu,\alpha(\bu)(x))$ to $[x,\bu,0]$.
Given $f\in C_c(X\times\Z^d)$, define a kernel $\bar{f}\in C_c(X\times\Z^d\times\Z^d)$ by the formula
\[
\bar{f}(x,\mathbf{n},\mathbf{m})=e^{\pi i\varphi_{\mathbf{m}}(\mathbf{n}-\mathbf{m})}f(\mathbf{n}-\mathbf{m},\alpha(\mathbf{n})(x)).
\]
Noting that $\exp(\pi i\varphi_{\mathbf{n}}(\mathbf{t}))=\sigma(-\mathbf{n},\mathbf{t})$ for $\mathbf{t}\in\Z^d$, one verifies that the assignment $f\mapsto\bar{f}$ defines a $*$-isomorphism $A^\Theta\to C^*(\cH,\Theta)$ with inverse given by $g\mapsto\underline{g}$, where
\begin{align*}
	\underline{g}(\mathbf{n},x)&=g(\alpha(-\mathbf{n})x,\mathbf{n},0).\qedhere
\end{align*}
\end{proof}

\subsection{Classifiability and the Elliott invariant}
Let us return to the question raised in the introduction. Suppose, as before, that we have two free and minimal actions $\alpha$ and $\beta$ of $\Z^d$ on $X$, and a map of mapping tori $F \colon Y_\alpha \to Y_\beta$ inducing an isomorphism of $K$-groups
\begin{equation}
\label{eq K iso}
K_*(A_\alpha)\cong K_*(A_\beta),\qquad K_*(A_\alpha^\Theta)\cong K_*(A_\beta^\Theta)	
\end{equation}
(the precise conditions on the map $F$ relating $Y_\alpha$ and $Y_\beta$ will be given later).
This suggests using the classification theory of $C^*$-algebras (the Elliott program) to lift this to an isomorphism of $C^*$-algebras. 

Recall that for any unital, stably finite $C^*$-algebra $B$, the \emph{Elliott invariant} is the tuple 
\begin{equation}\label{eq:eltuple}
(K_0(B),K_0(B)^+,[1_B],K_1(B), T(B),R),
\end{equation}
where the first four entries together comprise the ordered $K$-theory of $B$, $T(B)$ is the trace simplex (a Choquet simplex) and $R\colon T(B)\times K_0(B)\to \R$ is the pairing given by evaluating a trace at a projection. The class $[1_B]$ is referred to as the order unit.

In particular, it is known that:
\begin{theo}\label{thm:classification}
	  All unital, simple, separable, nuclear (stably finite, non-elementary) $C^*$-algebras that absorb the Jiang-Su algebra $\mathcal{Z}$ and satisfy the UCT are classified by the Elliott invariant. 
\end{theo}	 

This is a deep theorem which is the result of work by several authors; the interested reader should consult \cite{cetww:nucdim,egln:two,gln:one,gln:two,tww:quasidiag,win:classprod} and the references therein.

To ensure that the algebra $A^\Theta$ falls within this class, we make the following standing assumption.

\begin{assumption*}
The space $X$ has finite covering dimension.
\end{assumption*}

An important example that is of major interest in topological dynamics is when $X$ is a Cantor space. In this case, the covering dimension is zero. 

Let us notice that $A^\Theta$ is unital and separable because $X$ is compact and metrizable. The minimality of the action $\alpha$ of $\Z^d$ on $X$ implies this $C^*$-algebra is simple \cite[Corollary after Theorem 1]{archspiel:min}. It follows from the above assumption and the main theorem in \cite{szabo:rokhlin}, combined with \cite[Theorems B \& C]{gwz:rokhlin}, that $A^\Theta$ has finite nuclear dimension (in particular, it is nuclear). Further, $A^\Theta$ absorbs the Jiang-Su algebra $\mathcal{Z}$ by the main result in \cite{winter:pure}, together with Corollary 6.3 in the same paper. It also satisfies the UCT (see \cite[Lemma 10.6]{Tu99} or \cite[Proposition 9.5]{nestmeyer:loc}). Hence $A^\Theta$ is classifiable. 

We note here that it is possible for $A^\Theta$ to be $\mathcal{Z}$-stable without necessarily having the assumption on the covering dimension of $X$. For instance, if the action $\alpha$ has mean dimension zero, then the crossed product absorbs the Jiang-Su algebra and it is classifiable (see \cite{niu:meanzero}, the case of a single homeomorphism was shown in \cite{elniu:meanzero}).


 
In addition, by \cite{ror:stablereal}, $\mathcal{Z}$-stability implies the strict comparison property (of positive elements with respect to bounded traces). Since $\Z^d$ is amenable, $X$ admits an $\alpha$-invariant probability measure. From this, Lemma \ref{lem:tracemeas} below shows how to induce a tracial state on $A^\Theta$; in particular, $A^\Theta$ is stably finite \cite[Corollary 4.7]{cuntz:dimfunc}. For a stably finite $C^*$-algebra $B$, the Murray-von Neumann semigroup $V(B)$ admits an order-embedding into the Cuntz semigroup \cite{apt:clas}. In particular, this  means that the order on even $K$-theory is determined by the pairing $R$ (hence $K_0(B)^+$ can be dropped in Eq. \eqref{eq:eltuple}).

\begin{rema}
Regardless of classifiability, Phillips \cite{phil:cantorzd} shows that when $X$ is a Cantor space equipped with a $\Z^d$-action, the cone of the associated crossed product $K_0(C(X)\rtimes \Z^d)^+$ is determined by normalised traces.	
\end{rema}

Let us describe the relationship between invariant Borel measures on $X$ and tracial states on
$A_\alpha$. Since $\alpha$ is a free action, $\G$ is a principal \'etale groupoid with a Haar system given by counting measures.

\begin{lemm}[{\normalfont\cite[Proposition 5.4]{ren:group}}]\label{lem:tracemeas}
Let $\mu$ be a $\Z^d$-invariant Borel probability measure on $X$. For a function $f\in C_c(\G)$, the formula 
\begin{equation}\label{eq:dualtrace}
\tau_\mu(f)=\int_X f|_X\, d\mu
\end{equation}
defines a tracial state on $C^*(\G)$. Every tracial state on $C^*(\G)$ arises in this way.
\end{lemm}

With $\mu$ as above, let $\tilde{\mu}$ be the measure on $Y$ defined locally by $\mu\times \lambda^d$,
where $\lambda^d$ is the standard Lebesgue measure on $\R^d$. Then $\tilde{\mu}$ is $\tilde{\G}$-invariant (or, equivalently in our case, $\R^d$-invariant). 

Associated to $\tilde{\mu}$, there is a dual trace $\tilde{\tau}_\mu$ on $C^*(\tilde{\G})$, defined analogously to \eqref{eq:dualtrace}. 
It can be proved that $\tilde{\tau}_\mu$ is induced from $\tau_\mu$ under the trace induction operation in
\cite{comzettl:traceMorita}; the verification boils down to the equality
$\tau_\mu(\bra{\xi}\ket{\zeta})=\tilde{\tau}_\mu((\zeta\!\!\mid\! \xi))$, where $\xi,\zeta$ belong to the
$C^*(\G)$-$C^*(\tilde{\G})$ imprimitivity bimodule from Lemma \ref{lem:morita}, whose left and right inner products are $\bra{\xi}\ket{\zeta}$ and $(\zeta\!\!\mid\! \xi)$ respectively.
This has the following consequence.

\begin{lemm}[{\normalfont \cite[Proposition 2.1]{putkam:gap}}]\label{lem:traces}
The following diagram commutes:
\[
\begin{tikzcd}
K_0(C^*(\G)) \arrow[r, "\cong"] \arrow[d, "\tau_\mu"]      & K_0(C^*(\tilde{\G})) \arrow[d, "\tilde{\tau}_\mu"] \\
\R \ar[-,double line with arrow={-,-}]{r} & \R.               
\end{tikzcd}
\]
\end{lemm}

The twisted analogues of Lemmas \ref{lem:morita} and \ref{lem:traces} follow from Proposition \ref{prop kernel iso} after making the obvious adjustments. Let us summarise this as follows. Let $C^*(\G,\Theta)$ and $C^*(\tilde{\G},\Theta)$ be as in \eqref{eq A theta} and Definition \ref{def c g tilde theta}. The formula \eqref{eq:dualtrace} still defines a tracial state $\tau_\mu$ on $C^*(\G,\Theta)$, and $\tilde\tau_\mu$ on $C^*(\tilde{\G},\Theta)$. Recall the definition of $N$ from \eqref{eq N}. Then:
\begin{prop}
The subspace $s^{-1}(N)\subseteq\bar{\G}$ can be completed to an imprimitivity bimodule between $C^*(\cH,\Theta)$ and $C^*(\bar{\G},\Theta)$, hence to one between $C^*(\G,\Theta)=A^\Theta$ and $C^*(\tilde{\G},\Theta)$. The resulting strong Morita equivalence is compatible with traces, in the sense that we have a commuting diagram:
\begin{equation*}
\begin{tikzcd}
K_0(C^*(\G,\Theta)) \arrow[r, "\cong"] \arrow[d, "\tau_\mu"]      & K_0(C^*(\tilde{\G},\Theta)) \arrow[d, "\tilde{\tau}_\mu"] \\
\,\,\,\R\,\,\, \ar[-,double line with arrow={-,-}]{r} & \,\,\,\R\,.            
\end{tikzcd}
\end{equation*}
\end{prop}

Let us note that the result above (along with the general setup for twisted crossed product that we described) is implicit in \cite[Section 3]{benmat:magzero}.

\section{Identification of the order unit}\label{sec:two}
We now show that the isomorphism $K_*(A_\alpha)\cong K_*(A_\beta)$ from \eqref{eq K iso} preserves order units.
To do this, we will consider this isomorphism $K_*(A_\alpha)\cong K_*(A_\beta)$ from the point of view of two associated constructions:
\begin{itemize}[$\bullet$]
\item the mapping torus $Y =X\times_{\Z^d}\R^d$ associated to the $\Z^d$-action on $X$, and
\item the $C^*$-bundle $E=\R^d \times_{\Z^d} C(X) \to \T^d$ associated to the principal $\Z^d$-bundle $\R^d\to \T^d$.
\end{itemize}
\begin{rema}
The bundle $E_\alpha$ (and similarly for $E_\beta)$ is constructed as follows. The action $\alpha\colon\Z^d\to\mathrm{Homeo}(X)$ induces a map $\tilde{\alpha}\colon\Z^d\to\mathrm{Aut}(C(X))$. If $\lambda_{ij}\colon U_i\cap U_j\to \Z^d$ are transition functions for $\R^d\to \T^d$, then $\tilde\alpha\circ\lambda_{ij}$ are transition functions for $E_\alpha$. 
\end{rema}

The rich structure encoded in these two objects allows us to perform several $K$-theoretic manipulations that
simplify the problem. More precisely, we will reduce the problem to checking the commutativity of certain diagrams involving only basic $K$-theory elements and certain geometrically significant subspaces of $X$, such as the leaf, fibre, and base of the structure being considered.

\subsection{Reduction of twisted case}
Before continuing, let us turn to the twisted case and show how it reduces to the untwisted case, so that we may avoid repeating proofs with only minor variations. 

In essence, this is possible because any product involving the unit is constant along $A^{t\Theta}_{}$ ($t\in[0,1]$), so we would expect the isomorphism $K_0(A)\cong K_0(A^\Theta)$ to preserve the order unit. In order to give a rigorous proof, we need to first make some preparations, in the process establishing some useful notation for the rest of the paper.

In Section \ref{sec:prelims}, we briefly recalled Connes' Thom isomorphism, which in our setting gives an isomorphism $K_{*+d}(C(Y))\cong K_*(C^*(\tilde{\G}))$.
Let us explain how this arises through a $\KKK$-equivalence, by a result from \cite{skandalis:fack} (see also \cite{higkas:bc,nestmeyer:loc,val:shi} for the Dirac and Bott elements, and the Dirac-dual-Dirac method) . Denote the standard Dirac morphism on $\R^d$ by $D\in\KKK^{\Z^d}_d(C_0(\R^d),\C)$, and let $S^d$ denote  the $d$-fold suspension. 

\begin{convention*}
	We consider the algebra $S^d C(X)=C_0(X)\otimes C_0(\R^d)$ as a $\Z^d$-algebra under the \emph{diagonal} action; the action is \emph{not} trivial on the second tensor factor.

	Below, we use the notation $\jmath$ for Kasparov's descent morphism \cite[p. 172]{kas:descent}. We need not specify whether it's the full or reduced version, as it will only be used in cases where these two crossed products agree.
\end{convention*}

Let $1_{C(X)}$ denote the unit element in $\KKK^{\Z^d}(C(X),C(X))$. Applying Kasparov's descent map to
\begin{equation}
\label{eq DX}
D_X=D\hot 1_{C(X)}\in\KKK^{\Z^d}_d(S^dC(X),C(X)),
\end{equation}
we obtain a morphism $\jmath(D_X)\in \KKK(S^dC(X)\rtimes \Z^d,C^*(\G))$. Now Connes' Thom isomorphism can be expressed as the Kasparov product
\begin{equation}
\label{eq t}
t\coloneqq(M^{-1}\hatotimes_{S^dC(X)\rtimes \Z^d}\,\jmath(D_X) \,\hatotimes_{C^*(\G)}  \,m)\hatotimes-\colon K_{*+d}(C(Y))\to K_*(C^*(\tilde{\G})),
\end{equation}
where 
\begin{equation}
\label{eq M}
	M\in\KKK(S^dC(X)\rtimes \Z^d,C(Y)) 
\end{equation}
is the $\KKK$-class represented by Rieffel's imprimitivity bimodule
\cite[Corollary 1.7]{rieffel:proper} for the free and proper action of $\Z^d$ on $S^dC(X)$, and $m\in \KKK(C^*(\G),C^*(\tilde{\G}))$ is induced
by the imprimitivity bimodule of Lemma \ref{lem:morita}.

Fix an orientation-preserving homeomorphism $\R^d\cong(0,1)^d$, and let 
\begin{equation}
\label{eq f}
h\colon C_0(X\times \R^d)\to C_0(X\times (0,1)^d)	
\end{equation}
be the induced map. Let $b$ be the Bott element on $\R^d$, and let
\begin{equation}
\label{eq bX}
b_X=1_{C(X)}\hot b\in\KKK^{\Z^d}_d(C(X),S^dC(X)),
\end{equation}
so that $b_X\hot_{S^dC(X)} D_X = 1_{C(X)}$. For convenience, let us also define
\begin{equation}
\label{eq B}
B\coloneqq b_X\hot [h] \in \KKK_d^{\Z^d}(C(X),C_0(X\times (0,1)^d)).
\end{equation}
The fact that $D_X$ and $b_X$ are $\KKK$-equivalences and inverses of each other is a special case of the Dirac-dual-Dirac approach to the Baum--Connes conjecture, which is known to hold for the group $\Z^d$. For more details, in addition to \cite{higkas:bc,nestmeyer:loc,val:shi}, the reader may also see \cite[Section 3.9]{higgue:groupkth}.
\begin{lemm}\label{lem:twistunit}
The isomorphism $\phi\colon K_0(A)\to K_0(A^\Theta)$ in \eqref{eq twisted untwisted} takes $[1_{A}]$ to $[1_{A^\Theta}]$.
\end{lemm}
\begin{proof}
As explained in \cite[Section 4]{parae:twisttwo}, the isomorphism $\phi\colon K_0(A)\to K_0(A^\Theta)$ can be understood via an isomorphism of two $C^*$-bundles $E^0$ and $E^1$, which we now recall. Having done this, it will suffice to show that the unit $1_{E^0}\in K_0(\Gamma(E^0))$ is mapped to the unit $1_{E_1}\in K_1(\Gamma(E^1))$, where $\Gamma(E^i)$ is the $C^*$-algebra of sections of $E^i$, $i=0,1$.

To construct $E^i$, we first use the Packer-Raeburn stabilisation trick \cite[Theorem 3.4]{parae:twistone} to eliminate the twist associated to $\sigma$ as follows. Let $\cK$ denote the algebra of compact operators and $\alpha$ the $\Z^d$-action on $A$. For each $t\in[0,1]$, there exists an action 
$$\gamma^t=\gamma^t(\alpha,\sigma^t)\colon \Z^d\to \mathrm{Aut}(C(X)\otimes \cK(L^2(\Z^d)))$$ 
and a $C^*$-isomorphism 
\begin{equation}
\label{eq:prstab}
(C(X)\otimes\cK(L^2(\Z^d))) \rtimes_{\alpha\otimes 1,\sigma^t} \Z^d \cong (C(X)\otimes\cK(L^2(\Z^d)))\rtimes_{\gamma^t} \Z^d,
\end{equation}
where on the left-hand side the $\Z^d$-action is trivial on $\cK$, so that it is isomorphic to $C(X) \rtimes_{\alpha,\sigma^t} \Z^d\otimes\cK$. For convenience, let us set $B_X=C(X)\otimes \cK(L^2(\Z^d))$. By \cite[Lemma 3.3]{parae:twistone}, the isomorphism \eqref{eq:prstab} preserves the natural inclusion of $B_X$ into the crossed product. For each $t$, define 
$$E^t=\R^d\times_{\Z^d} B_X \to\T^d,$$
where $\Z^d$ acts on $B_X$ via $\gamma^t$. 
Thus if
$\lambda_{ij}\colon U_i\cap U_j\to \Z^d$ are the transition functions for $\R^d$ (viewed as a principal $\Z^d$-bundle over $\T^d$), then
$\gamma^t\circ\lambda_{ij}$ are the transition functions for $E^t$, where $\gamma^t=\gamma(\alpha,\sigma^t)$. The family $\{E^t\}_{t}$ fits together into a bundle 
$$e\colon E^{[0,1]}\to \T^d\times [0,1],$$ 
giving an isomorphism $E^0\xrightarrow{\cong}E^1$ of $C^*$-bundles (see \cite[Section 3, Corollary 4.6]{huse:fb}). Denote the resulting $K$-theoretic isomorphism by
\begin{gather}
\label{eq:isobundles}
\psi\colon K_d(\Gamma(E^0))\to K_d(\Gamma(E^1)).
\end{gather}

Next, to make the connection with the algebras $A$ and $A^\Theta$, we use the induction functor from $\Z^d$ to $\R^d$. For a $C^*$-algebra $P$ with a $\Z^d$-action $\mu$, this is defined by
\begin{gather}
\label{eq:tctw1}
\Ind_{\Z^d}^{\R^d}(P,\mu)= (C_0(\R^d)\otimes P) \rtimes_{\text{rt}\otimes\mu} \Z^d,
\end{gather}
where $\text{rt}$ is the action of $\Z^d$ on $C_0(\R^d)$ by right-translation. (The main result in \cite{brown:proper}, implies that this model for induction used is Morita equivalent to the more familiar one described as a generalised fixed-point algebra.) By \cite[Proposition 2.1]{parae:twistthree}, we have isomorphisms
\begin{align*}
\tau_0\colon K_0(A)\cong K_d(\Ind_{\Z^d}^{\R^d}(B_X,\gamma^0))&=K_d((C_0(\R^d)\otimes B_X)\rtimes_{\text{rt}\otimes \gamma^0} \Z^d),\\
\tau_1\colon K_0(A^\Theta)\cong K_d(\Ind_{\Z^d}^{\R^d}(B_X,\gamma^1))&=K_d((C_0(\R^d)\otimes B_X)\rtimes_{\text{rt}\otimes \gamma^1} \Z^d).
\end{align*}
By \cite[Lemma 4.3]{parae:twisttwo}, we also have isomorphisms
\begin{align}
\label{eq:omega}
\omega_t\colon K_d((C_0(\R^d)\otimes B_X)\rtimes_{\text{rt}\otimes\gamma^t} \Z^d)\cong K_d(\Gamma(E^t)),
\end{align}
for each $t\in[0,1]$. For $i=0,1$, let $\rho_i=\omega_i\circ\tau_i$. Then by \cite[Section 4]{parae:twisttwo}, we have a commutative diagram
\[
\begin{tikzcd}
K_0(A)\ar{r}{\rho_0}\ar{d}{\phi} & K_d(\Gamma(E^0))\ar{d}{\psi}\\
K_0(A^\Theta)\ar{r}{\rho_1} & K_d(\Gamma(E^1)).	
\end{tikzcd}
\]

Thus it remains to show that $\psi(1_E^0)=1_E^1$. By Connes' Thom isomorphism, 
\begin{gather}
\label{eq:tctw2}
K_0(\R^d\ltimes \Ind_{\Z^d}^{\R^d}(B_X,\gamma^t))\cong K_d(\Ind_{\Z^d}^{\R^d}(B_X,\gamma^t)).
\end{gather}
The algebra $\R^d\ltimes \Ind_{\Z^d}^{\R^d}(B_X,\gamma^t)$ can be rewritten as
\begin{equation}
\label{eq K BX}
((\R^d\ltimes C_0(\R^d)) \otimes B_X) \rtimes \Z^d \cong \cK(L^2(\R^d))\otimes (B_X\rtimes_{\gamma^t}\Z^d),
\end{equation}
where the isomorphism holds because the induced action $\Z^d\to \mathrm{Aut}(\cK(L^2(\R^d)))$ has trivial Mackey obstruction \cite[Proposition D~27]{danawil:crossprod}. Let $p$ and $q$ be minimal projections in $\cK(L^2(\Z^d))$ and $\cK(L^2(\R^d))$ respectively, so that $[q]\otimes
[1_{C(X)}\otimes p]$ is the unit in $ K_0(\cK(L^2(\R^d))\otimes (B_X\rtimes_{\gamma^t}\Z^d))$ for each $t$. Eq. \eqref{eq K BX} implies that this element is mapped to $b\otimes [1_{C(X)}\otimes p]\in K_d(\Ind_{\Z^d}^{\R^d}(B_X,\gamma^t))$ under \eqref{eq:tctw2}, where $b\in K_d(C_0(\R^d))$ is the Bott element.
%
The maps $\omega_t$ in \eqref{eq:omega} then send $b\otimes[1_{C(X)}\otimes
p]$ to $1_{E^t}\in K_d(\Gamma(E^t))$. 

We can realize $1_{E^0}$ and $1_{E^1}$ explicitly as follows. Take an open subset $U\subseteq \T^d$ that locally trivializes both $E^0$ and $E^1$ and for which there exists a local trivialization $h\colon U\times [0,1] \times B_X \xrightarrow{\cong} e^{-1} (U\times [0,1])$ of $E^{[0,1]}$. 
Let us write $b$ as $[b_U+e_{11}] - [e_{11}]$, with $b_U$ supported inside $\pi^{-1}(U)$ where $\pi\colon\R^d\to\T^d$ is the canonical quotient map, and $e_{11}$ is the standard matrix unit rank-one projection. By the preceding discussion, $1_{E^0}$ and $1_{E^1}$ are represented (omitting $e_{11}$) respectively by 
\begin{align*}
b_U(\cdot)\cdot h(\cdot\,,0, 1_{C(X)}\otimes p)\quad\textnormal{ and }\quad
b_U(\cdot)\cdot h(\cdot\,,1, 1_{C(X)}\otimes p).
\end{align*}
By construction of the bundle isomorphism $E^0\cong E^1$ (see the proof of \cite[Section 3, Theorem 4.3]{huse:fb}), $\psi$ takes one to the other. This completes the proof.
\end{proof}

\subsection{Identification of unit (untwisted case)}
\label{subsec untwisted}
In this section we identify the unit inside the mapping torus under the isomorphism $K_0(A)\cong K^{d}(Y)$.
By Lemma \ref{lem:twistunit}, it suffices to consider $A$ instead of $A^\Theta$.

We begin by observing that the mapping torus $Y$ can be identified with a quotient of the cylinder $X\times [0,1]^d$. Let $Y^0=X\times (0,1)^d$, viewed as a common subset of $Y_\alpha$ and $Y_\beta$. 
Then $Y^0$ is a ``basic block'' in the structure of the mapping tori, in the sense that any difference in structure between $Y_\alpha$ and $Y_\beta$ must lie away from $Y^0$ and arise only from the way in which the boundary components of $X\times [0,1]^d$ are glued. Considering the map \eqref{eq f}, we can speak of the Bott element $b_0\in K^d(Y^0)$.

\begin{defi}\label{def:basic}
Let $\iota\colon C_0(Y^0)\to C(Y)$ denote the open inclusion. Consider the following (not necessarily commutative) diagram:
\begin{equation}
\label{eq basic diagram}
	\xymatrix{K^d(Y^0) \ar[r]^-{\iota_{\beta *}} \ar[dr]_-{\iota_{\alpha *}} & K^d(Y_\beta) \ar[d]^-{F^*} \\
	& K^d(Y_\alpha).}
\end{equation}
We say that a map $F\colon Y_\alpha \to Y_\beta$ is \emph{basic} if $F^* \iota_{\beta *}(b_0)=\iota_{\alpha *}(b_0)$.
\end{defi}

The next lemma provides a useful factorisation of $\iota$ in terms of $\KKK$-elements.
\begin{lemm} 
\label{lem:fac}
Let $\lambda\colon C_0(X\times (0,1)^d)\to C_0(X\times \R^d)$ be the inclusion given by extending functions by zero on the complement of $X\times (0,1)^d$ in $X\times \R^d$. Let $M$ be as in \eqref{eq M}. Then
\begin{gather*}
[\iota] = [\lambda]\hot_{S^dC(X)} I\hot_{S^dC(X)\rtimes \Z^d}\,M,
\end{gather*}
where the element $I\in\KKK(S^dC(X),S^dC(X)\rtimes \Z^d)
$ is induced by the standard embedding at $0\in\Z^d$, i.e., $a\mapsto a\delta_0$.
\end{lemm}
\begin{proof}
The $C^*$-correspondence $I\hot M$ is defined as a completion of $C_c(X\times \R^d)$. The unital $C^*$-algebra $C(Y)$ acts from
the right as bounded continuous functions via pullback along the quotient map $p : X\times \R^d \to Y$. The full
$C(Y)$-valued inner product on $C_c(X\times \R^d)$ is defined as 
\begin{equation*}
\bra{\xi}\ket{\zeta}(y) =\sum_{p(x) = y} \bar{\xi}(x)\zeta(x),
\end{equation*}
where $\xi, \zeta$ are in $C_0(X\times \R^d)$ and $y \in Y$. Define the bimodule underlying $M$ to be the completion of
$C_c(X\times\R^d)$ with respect to the induced norm. The left action of $S^dC(X)$ is given by pointwise multiplication. The full $S^dC(X) \rtimes \Z^d$-valued left inner product is defined as
\begin{equation*}
(\xi|\zeta)(\mathbf{n})=\bar{\xi}(\mathbf{n}\cdot \zeta)
\end{equation*}
with $\xi,\zeta \in C_c(X\times \R^d)$, giving a function in $C_c(\Z^d,S^dC(X))$. Now given $f\in C_c(X\times (0,1)^d)$, we
express $I(\lambda(f))$ as $(\xi|\zeta)$ (\emph{left} inner product) by choosing $\xi=\bar{f}$ and $\zeta$ such that $\xi\zeta=\xi$ (note that
$\bar{\xi}(\mathbf{n}\cdot \zeta)=0$ for any nontrivial vector $\mathbf{n}\in\Z^d$). Now notice that computing the \emph{right} inner product, we
get that $\bra{\xi}\ket{\zeta}=\iota(f)$. Having this, a routine argument shows that $\iota^*[C(Y)] \cong  (I\circ \lambda)^*[M]$ (cf. \cite[Prop.~18.3.6]{black:kth}).
\end{proof}

Let $m\in \KKK(C^*(\G),C^*(\tilde{\G}))$ denote the Morita equivalence element from Lemma \ref{lem:morita}, and let
\begin{equation}
\label{eq i}
i\colon C(X)\hookrightarrow C^*(\G)	
\end{equation}
be the unit space inclusion. We refer to the map
\begin{equation}
\label{eq transfer}
i^*(m)\coloneqq i^*(m)\hot - \colon K_0(C(X)) \to K_0(C^*(\tilde{\G})),
\end{equation}
as the \emph{transfer map}, since it transports projections living on the transversal of the foliation to the associated foliation $C^*$-algebra (cf.  \cite[Chapter 2, Section 8]{connes:ncg},\cite[Section 3]{putkam:gap}).

Following Connes \cite[p. 126]{connes:ncg}, 
we can give an explicit formula for the map $i^*(m)$. Let $N$ be the transversal \eqref{eq N}. Let $V=Y\times (-\epsilon,\epsilon)^d \subseteq \tilde{\G}$ be an open neighbourhood of the unit space. We may take $\epsilon$ small enough so that if $\gamma\in V$ with $s(\gamma)\in N$ and $r(\gamma)\in N$, then
$\gamma\in Y=\tilde{\G}^{(0)}$. Take $\eta\in C_c(r^{-1}(N))$ with $\mathrm{supp}(\eta)\subseteq V$ and $\int \lvert
\eta(\gamma)\rvert^2=1$, where the integral is taken over $\{\gamma\in \tilde{\G}\mid r(\gamma) =y\}$ for any fixed $y\in N$. For
instance, given a compactly supported function $f\colon (-\epsilon,\epsilon)^d\to \R$ with $\int \lvert f\rvert ^2=1$, we can set
$\eta(\gamma)=\xi(y,r)=f(r)$.

We can then give a concrete description of the element $i^*(m)([1_X])$ as
\begin{gather*}
	i^*(m)([1_X]) = [e]\in K_0(C^*(\tilde{\G})),\\
e(\gamma)  = \sum_{\substack{s(\gamma^\prime)=s(\gamma)\\ r(\gamma^\prime)\in N}}  \bar{\eta}(\gamma^\prime \gamma^{-1})\eta(\gamma^\prime).
\end{gather*}

\begin{lemm}\label{lem:trlem}
We have $i^*(m)([1_X])=[e]$.
\end{lemm}
\begin{proof}
The element $i^*(m)([1_X])$ is represented by the right Hilbert $C^*(\tilde{\G})$-module $m$, ignoring the left action. We show
	that the idempotent associated to this module is precisely $e$. To do this, we find a frame \cite[Theorem 4.1]{frank:frames}, i.e. a
	sequence $\{\xi_i\}$ such that for any element $f$ in the module, we have $f=\sum_i\xi_i\bra{\xi_i}\ket{f}$. Then $e$ should take the form $(e)_{ij}=\bra{\xi_i}\ket{\xi_j}$. 
	
	Recall that by
	definition $m$ is represented by the completion of $C_c(r^{-1}(N))$ with respect to the inner
product
\[
\bra{\xi}\ket{\zeta}(\gamma)= \sum_{\substack{s(\gamma^\prime)=s(\gamma)\\ r(\gamma^\prime)\in N}}\bar{\xi}(\gamma^\prime \gamma^{-1})\zeta(\gamma^\prime).
\]
Notice that $\bra{\eta}\ket{\eta}=e$, which suggests that $\{\eta\}$ is a frame. Indeed, for any $f\in C_c(r^{-1}(N))$, we have:
\begin{align*}
f(\gamma)&\,\,\,= \int_{\gamma_1\gamma_2=\gamma}\sum_{\substack{s(\gamma^\prime)=s(\gamma)\\ r(\gamma^\prime)\in N}} \eta(\gamma_1)\bar{\eta}(\gamma^\prime \gamma_2^{-1})f(\gamma^\prime) \\
& \underset{\gamma=\gamma^\prime}{=} \int \eta(\gamma_1)\bar{\eta}(\gamma_1)\cdot f(\gamma) = f(\gamma),
\end{align*}
where the second equality holds because $\gamma_1\gamma_2(\gamma^\prime)^{-1}$ belongs to $V$.
\end{proof}


We are now ready to state and prove an alternative description of the transfer map \eqref{eq transfer} in terms of simpler and more familiar morphisms. 

\begin{prop}\label{prop:transfer}
The transfer map factors as
$$i^*(m)=B \hot_{C(X\times (0,1)^d)} [\iota] \hot _{C(Y)}\, t,$$
where $t$, $B$, and $\iota$ are as in \eqref{eq t}, \eqref{eq B}, and Definition \ref{def:basic} respectively.
In particular, the following diagram commutes:
\[
\begin{tikzcd}
K_0(C(X)) \arrow[r, "i^*(m)"] \arrow[d, "B"]      & K_0(C^*(\tilde{\G})) \ \\
K_d(C(X\times (0,1)^d)) \ar[r, "\iota_*"] & K_d(C(Y))  \arrow[u, "t"].             
\end{tikzcd}
\]
\end{prop}
\begin{proof}
Expanding definitions and using Lemma \ref{lem:fac}, we see that the composition $B \hot [\iota] \hot t$ is given by
\begin{equation}\label{eq:mcomp}
b_X\hot [h] \hot [\lambda] \hot I \hot M \hot M^{-1}\hot\jmath(D_X)\hot m,
\end{equation}
where $b_X$ is the Bott element \eqref{eq bX}, $h$ is the homeomorphism \eqref{eq f}, $\lambda$ is the open inclusion $C_0(X\times (0,1)^d)\to C_0(X\times \R^d)$, and $I\in\KKK(S^dC(X),S^dC(X)\rtimes \Z^d)$ is induced by the standard embedding at $0\in\Z^d$ ($a\mapsto a\delta_0$).
As $\lambda\circ h$ is properly homotopic to the identity, we have $[h] \hot [\lambda]=1_{\R^d}$, so Eq. \eqref{eq:mcomp} equals
\[
b_X \hot I \hot \jmath(D_X)\hot m.
\]
Let $[i]\in\KKK(C(X), C^*(\G))$ be the element corresponding to the inclusion \eqref{eq i}. We see it suffices to show that $b_X \hot I = [i] \hot \jmath(b_X)$. Write $b_X=[\pi,H,F]$. Then 
\[
[i] \hot \jmath(b_X)=[\pi \otimes 1, H\otimes_{C(X)} C(X)\rtimes \Z^d, F\otimes 1],
\]
which by definition of the Kasparov product equals $b_X\hot I$.
\end{proof}

\begin{rema}
The combination of Proposition \ref{prop:transfer} and Lemma \ref{lem:trlem} above provides an alternative proof of \cite[Proposition 3.2]{putkam:gap}.
\end{rema}

We are now ready to prove the main results of this section. We will consider maps $F\colon Y_\alpha \to Y_\beta$ between mapping tori that satisfy at least one of the following properties:
\begin{itemize}[$\bullet$]
	\item $F$ is basic in the sense of Definition \ref{def:basic};
	\item $F$ is an isomorphism of bundles over $\T^d$ with fibre $X$; or
	\item $F$ is an isomorphism of foliated spaces.
\end{itemize}

Let us recall that a foliated space structure on $Y_\alpha$ is roughly given by a (countable) atlas of open charts $\{(U_i,\psi_i)\}$ of
$Y_\alpha$ such that $\psi_i \colon U_i \to X\times \R^d$ is a homeomorphism and $\psi_j\circ
\psi_i^{-1}(x,\mathbf{r})=(\Psi^1(x),\Psi^2(x,\mathbf{r}))$ with $\Psi^2(x,\cdot)$ a diffeomorphism for any $x$ (analogously for $Y_\beta$). For more details on the definition, we refer to \cite[Chapter 2]{mooscho:globanalysis}. 

Note that $Y^0\subseteq Y_\alpha$ is a possible choice of
chart $(U,\psi)$, where the map $\psi$ is essentially given by fixing a smooth map $(0,1)^d\cong \R^d$. By contrast, the change of coordinates for the fibre bundle takes the form $\psi_j\circ \psi_i^{-1}(x,\mathbf{r})=(\Psi^1(x,\mathbf{r}),\Psi^2(\mathbf{r}))$, so that the fibre is allowed to transform depending on the basepoint.

Define $\Phi\colon K_0(A) \to K^d(Y)$ to be the composition
\begin{equation}\label{eq:phidef}
\xymatrixcolsep{3pc}\xymatrix{K_0(A) \ar[r]^-{m} & K_0(C^*(\tilde{\G})) \ar[r]^-{t^{-1}} & K^d(Y)},
\end{equation}
where $m$ is the Morita equivalence from Lemma \ref{lem:morita}, and $t$ is as in \eqref{eq t}.

\begin{prop}\label{prop:presordunit}
Suppose the map $F\colon Y_\alpha\to Y_\beta$ is basic and induces an isomorphism on $K$-theory. Then the isomorphism $\Phi_\alpha^{-1}\circ F^*\circ \Phi_\beta\colon K_0(A_\beta)\to K_0(A_\alpha)$ preserves the order unit. In particular, this holds if $F$ is a basic homotopy equivalence.
\end{prop}
\begin{proof}
First note that $[1_{A_\beta}]=i_{\alpha*}[1_{C(X)}]$ and $[1_{A_\alpha}]=i_{\beta*}[1_{C(X)}]$, where $i_{\alpha*}$ and $i_{\beta*}$ are the maps on $K$-theory induced by unit space inclusions \eqref{eq i}. These maps fit into a diagram:
\[
\begin{tikzcd}
K_0(C^*(\tilde{\G}_\alpha)) \arrow[r, "t_\alpha^{-1}"]                      & K^d(Y_\alpha)                     &              \\
K_0(A_\alpha) \arrow[u, "m_\alpha"]                      & K^d(Y^0) \arrow[u, "\iota_\alpha*"] \arrow[r, "\iota_\beta*"] & K^d(Y_\beta) \arrow[lu, "F^*"'] \\
K_0(C(X)) \arrow[r, "i_\beta*"] \arrow[u, "i_\alpha*"] \arrow[ru, "B"] & K_0(A_\beta) \arrow[r, "m_\beta"]           & K_0(C^*(\tilde{\G}_\beta)) \arrow[u, "t_\beta^{-1}"],
\end{tikzcd}
\]
(recall that $B$ is induced by the Bott element, see Eq. \eqref{eq B}).
Using Proposition \ref{prop:transfer} and the assumption that $F$ is basic, a straightforward diagram chase shows that $\Phi_\alpha^{-1}\circ F^*\circ \Phi_\beta$ maps $[1_{A_\beta}]$ to $[1_{A_\alpha}]$. 
\end{proof}

The next proposition deals with the two remaining cases.

\begin{prop}\label{prop:either}
	Suppose $F$ is an orientation-preserving isomorphism of either the fibre bundle structure or the foliated space structure. Then $F$ is basic.
\end{prop}
\begin{proof}
	We need to show that $F^*(b_\beta)=(b_\alpha)$, where $b_\alpha=\iota_{\alpha*}(b)$ and $b=b_0$ is the Bott element as in Definition \ref{def:basic}. Note that $b(x,\mathbf{r})=b(\mathbf{r})$, i.e. $b$ does not depend on $X$. We may suppose that the support of
	$b$ is contained in $F(Y^0_\alpha)\cap Y^0_\beta$, since $Y^0_\alpha$ and $Y^0_\beta$ are copies of the
	same space, and we may view $b$ as a formal difference $[b+e_{11}] -[e_{11}]$, where $b+e_{11}$ is a function of the Alexandroff
	compactification $(0,1)^{d \sim}$, and $e_{11}$ is the standard rank-one projection (as in the proof of Lemma \ref{lem:twistunit}). 
	Then, in the case of fibre bundles, $F^*(b_\beta)(x,\mathbf{r})= b_\beta(F(x,\mathbf{r}))=b(F^2(\mathbf{r}))$ since
	$b_\beta$ does not depend on the fibre. Since orientation is preserved, we have $F^*(b_\beta)(x,\mathbf{r})= b_\beta$. 
	
	In the case of
	foliated spaces, we have $F^*(b_\beta)(x,\mathbf{r})=b(F^2(x,\mathbf{r}))$. For any fixed $y\in X$, let us write $b_y=(b+e_{11})\circ F^2(y,\cdot)
	\in C((0,1)^{d \sim})$. We have $[b_y] -[1]=b$ in $K^0((0,1)^d)$ because $F^2(y,\cdot)$ is orientation-preserving. For any
	other $x\in X$, there is a homotopy $h_x$ with $h_x(0)=b_x$ and $h_x(1)= b_y$. When $b_x$ is sufficiently close to $b_y$,  $h_x$
	depends continuously on $b_x$ (cf. \cite[Lemma 2.2.4]{rordam:Kth}). Thus the collection of $h_x$ together yields a homotopy between $(b+e_{11})\circ F^2$ and $1_X\otimes
	b_y$. This shows that $F^*(b_\beta)=(b_\alpha)$, as claimed.
\end{proof}


Before closing this section, we consider the setting where we are given an isomorphism of the associated $C^*$-bundles,
$F\colon E_\beta \to E_\alpha$. The main tool for handling this -- namely an induction procedure -- has already been used in the proof of Lemma
\ref{lem:twistunit} when we related the $K$-theory of $A_\alpha$ to the $K$-theory of the algebra of sections of $E_\alpha$.

\begin{prop}
	Suppose given an isomorphism of bundles $F\colon E_\beta \to E_\alpha$. Then we have a $K$-theory isomorphism
	$K_*(A_\alpha)\cong K_*(A_\beta)$ preserving the order unit.
\end{prop}
\begin{proof}
By stabilizing $A_\alpha$, we are able to write the isomorphism
	\begin{equation}\label{eq:find}
\cK(L^2(\R^d))\otimes C(X)\rtimes_{1\otimes \alpha} \Z^d \cong (\R^d\ltimes C_0(\R^d)) \otimes C(X) \rtimes \Z^d,
\end{equation}
	where on the right-hand side we may assume that $\Z^d$ acts by translation on $C_0(\R^d)$, since the action of $\Z^d$ on $\cK(L^2(\R^d))$ is trivial (the translation action of $\Z^d$ on $C_0(\R^d)$ induces an action $\Z^d\to \mathrm{Aut}(\cK(L^2(\R^d)))$ with trivial Mackey obstruction \cite[Proposition D~27]{danawil:crossprod}).
	
	Using Connes' Thom isomorphism and the induction functor, we can rewrite the $K$-group of the $C^*$-algebra on the right of \eqref{eq:find} as $K_{*+d}(\Ind_{\Z^d}^{\R^d}(C(X)))$, where $\Ind_{\Z^d}^{\R^d}$ is as in Eq.~\eqref{eq:tctw1}. By \cite[Lemma 4.3]{parae:twisttwo}, we have
\[
	K_{*+d}(\Ind_{\Z^d}^{\R^d}(C(X)))\cong K_{*+d}(\Gamma(E_\alpha)),
\]
where on the right-hand side we have the algebra of sections of the $C^*$-bundle $E_\alpha$.

The map $F$ induces an isomorphism $\Gamma(E_\alpha)\cong \Gamma(E_\beta)$, and hence a $K$-theoretic isomorphism
$K_*(C(X)\rtimes_\alpha \Z^d) \cong K_*(C(X)\rtimes_\beta \Z^d)$. The order unit corresponds to the element $b\otimes 1_X\in
K_d(\Gamma(E_\alpha))$, where $b$ is the Bott element in $K_d(C_0(U))$, and $U\subseteq \T^d$ is a trivializing open subset homeomorphic to $\R^d$. Viewing $\T^d$ as a quotient of $[0,1]^d$, we may take $U$ to be the intersection $(0,1)^d\cap
F^2((0,1)^d)$. Then the isomorphism $F$ is implemented by a continuous map $L\colon U\to \mathrm{Aut}(C(X))$, whence the element $b\otimes 1_X$ is preserved, as is the order unit of $K_0$.
\end{proof}

\section{Index theory and trace pairings}\label{sec:index}

In this section we lay out the fundamental index-theoretic constructions that allow us to control the trace pairings on
$A_\alpha=C^*(\G_\alpha)$ and its twisted version. 

The use of index theory to compute the trace values in our setting is quite natural, since our main hypothesis (i.e. identification of the homotopy quotients) directly involves the $K$-theoretic isomorphism given by Connes' Thom map, which is represented by the class of the Dirac operator. Indeed, working in a similar context, a number of other authors \cite{benmat:magzero,benoy:gap,putkam:gap} have successfully applied the index theorem for foliated spaces, within the framework of \emph{tangential} index theory. We will adopt a similar setup, which we now recall for the reader's convenience.

If $Y$ is a foliated space with $d$-dimensional leaves, then there is an induced foliated space
structure on the holonomy groupoid $\tilde{\G}$ whose leaves have dimension $2d$. In this sense, the source map can be viewed as
a ``submersion'', meaning that for each $y\in Y$ there exists an open neighbourhood $U\subset Y$ and open subsets
$\tilde{U}\subseteq\tilde{\G}$ and $V\subseteq \R^d$ such that $s|_{\tilde{U}}\colon \tilde{U}\to U$ is identified with the map $U\times V\to U$ via suitable coordinate charts $\psi\colon\tilde{U}\to U\times V$. We will use the notation $C^{0,\infty}$ to denote the space of tangentially smooth functions.

A tangential operator is a linear map $P\colon C^{0,\infty}_c(\tilde{\G}) \to C^{0,\infty}(\tilde{\G})$ such that there exists a family of
linear functions $\{P_y\colon C^{\infty}_c(\tilde{\G}_y)\to C^{\infty}(\tilde{\G}_y)\}_{y\in Y}$ satisfying
\begin{align*}
	(Pf)(\gamma) & =P_{s(\gamma)}(f|_{\tilde{\G}_{s(\gamma)}})(\gamma),\\
	[P_{s(\gamma_1)}f(\cdot\, \gamma_1^{-1})](\gamma) & =(P_{r(\gamma_1)}f|_{\tilde{\G}_{r(\gamma_1)}})(\gamma\gamma_1^{-1}).
\end{align*}
The second condition is the key invariance property of tangential operators. 

We are interested in tangential differential operators
whose family $\{P_y\}_{y\in Y}$ is \emph{indexed by} the source map $s\colon \tilde{\G}\to Y$, meaning that for any local
trivialisation $(U,\tilde{U},\psi)$ as above, the operator $P$ takes the form 
\begin{align*}
	P\colon C_c^{0,\infty}(U\times V)&\to C^{0,\infty}(U\times V)\\
	(Pf)(y,v)&=[P_yf(y,\cdot)](v),
\end{align*}
where $P_y$ is a classical (pseudo-)differential operator acting on $V$. The invariance property guarantees that $P$ admits a
distributional kernel $k_P$ such that the kernel associated to $P_y$ is of the form
$k_y(\gamma,\gamma^{\prime})=k_P(\gamma\gamma^{\prime -1})$ \cite[Section 5]{nwx:pseudo}. If $k_P$ belongs to $C_c^{0,\infty}(G)$, we say that $P$ is
\emph{compactly smoothing}. 

The discussion so far can be readily extended to operators acting on vector-valued functions. To further
extend to the case of sections of vector bundles, we make the following preparations. First of all, we note that $K$-theory
defined via tangentially smooth bundles coincides with the usual $K$-theory \cite[Proposition 2.16]{mooscho:globanalysis}. We lift bundles from $Y$ to
$\tilde{\G}$ by pulling back along the range map. If $E$ is a bundle over $Y$, we can find an orthogonal projection in
$M_N(C^{0,\infty}(Y))$ with image $E$ and kernel $E^\perp$. Composing with the range map, we get an element $e$ in the multiplier
algebra of $C^*(\tilde{\G})$, and a tangential differential operator acting on $E$ is defined to be of the form $P=eP^\prime e$
where $P^\prime$ is an operator on the trivial bundle.

Defining $(r^*Pf)(y)=[P_{s(\gamma)}(f\circ r)](\gamma)$ for any $\gamma$ in $\tilde{\G}^y$ gives a well-defined map (by invariance)
into operators on $Y$, preserving compactly smoothing operators, degree of pseudodifferential operators, and composition. The
principal symbol $\sigma(P)$ of $P$ is defined as the principal symbol of $r^*P$, yielding a map $\sigma(P)\colon S^*F\to
\pi^*(\mathrm{Hom}(E,E^\prime))$. Here, $F$ is the foliated tangent bundle, $S^*F$ is the cosphere bundle, and $\pi\colon
S^*F\to Y$ is the canonical projection. An operator is \emph{tangentially elliptic} if its principal symbol takes values in
$\mathrm{Iso}(E,E^\prime)$ rather than $\mathrm{Hom}(E,E^\prime)$.

\begin{rema}\label{rem:diffbundle}
	We are going to tacitly assume the existence of an isomorphism $E\cong E^\prime$ (since otherwise we wouldn't have an algebra structure for various terms of the exact sequence below). This is only a superficial obstacle, and it can be bypassed as explained on \cite[page 186]{mooscho:globanalysis}. The framework is much simplified in the concrete setting of the mapping torus relevant for this paper, so we will neglect this subtlety in the sequel.
\end{rema}
Degree-zero operators extend on a field of Hilbert spaces as $P_y\colon L^2(\tilde{\G}_y,\lambda_y,E|_{\tilde{\G}_y})\to
L^2(\tilde{\G}_y,\lambda_y,E^\prime|_{\tilde{\G}_y})$, allowing the definition of a formal adjoint and a norm. As a result, the
degree zero operators form a normed $*$-algebra which we can complete to a $C^*$-algebra, denoted by $\Psi^0(\tilde{\G},E,E^\prime)$. Under
this completion, the tangentially smooth sections $\Gamma(\tilde{\G},\mathrm{Hom}(E,E^\prime))$ form a $C^*$-ideal
$C^*_r(\tilde{\G},\mathrm{Hom}(E,E^\prime))$, which is stably isomorphic to the groupoid $C^*$-algebra $C^*(\tilde{\G})$. In
all, we have a short exact sequence
\[
	0 \longrightarrow C^*_r(\tilde{\G},\mathrm{Hom}(E,E^\prime)) \longrightarrow \Psi^0(\tilde{\G},E,E^\prime)
	\overset{\sigma}{\longrightarrow}
\Gamma(S^*F,\pi^*(\mathrm{Hom}(E,E^\prime))) \longrightarrow 0.
\]
In a standard way, we can use the associated $6$-term exact sequence in $K$-theory to define the index class $\Ind(P)$ of an elliptic operator. More precisely, if $\partial$ denotes the boundary map of this sequence, then
\begin{equation}
\label{eq index K}
\Ind(P)\coloneqq\partial([\sigma(P)])\in K_0(C^*(\tilde{\G})),
\end{equation}
where $[\sigma(P)]\in K_1(\Gamma(S^*F,\pi^*(\mathrm{Hom}(E,E^\prime))))$. 

For a differential operator $D$ of non-zero degree $m$ (which is the case for the leafwise
Dirac operator underlying the definition of $t\in\KKK_d(C(Y),C^*(\tilde{\G}))$), we take its degree-zero transform
\begin{equation}
\label{eq transform}	
	P=(1+\Delta)^{-m/2}D,
\end{equation}
where $\Delta$ is the standard (tangential) Hodge-Laplace operator \cite[p. 172]{mooscho:globanalysis}. This procedure preserves the principal symbol up to homotopy. In order to obtain a \emph{numerical} index from the tangential operator $P$, we will use (signed) \emph{invariant transverse} measures on $Y$. In our setting, such measures correspond precisely to
(signed) $\Z^d$-invariant measures on $X$ (see \cite[pp. 105-106]{mooscho:globanalysis}); see also Remark \ref{rem:coibvrema} below. We refer the reader to \cite[Chapter 4]{mooscho:globanalysis} for a more general discussion of transverse measures. Let $M(Y)$ denote the space of such measures. If $\mu\in M(Y)$ is a probability measure, then as explained in Section \ref{sec:prelims}, we have associated traces $\tilde{\tau}_\mu$ and $\tau_\mu$ on $C^*(\tilde{\G})$ and $A=C^*(\G)$ respectively.

\begin{defi}
	The \emph{analytic $\mu$-index} of $P$ is
	\[\Ind_\mu(P)\coloneqq\tilde{\tau}_\mu(\Ind(P)),\]
where $\Ind(P)$ is as in \eqref{eq index K}.
\end{defi}
Since $A$ possesses the strict comparison property, the positive cone of its $K$-theory is determined by the values of tracial states,
meaning that $x\in K_0(A)$ belongs to $K_0(A)^+$ if and only if $\tau(x)>0$ for any tracial state $\tau$ on $A$. In view of Lemma
\ref{lem:tracemeas}, any such trace $\tau$ is of the form $\tau=\tau_\mu$ for an invariant measure $\mu$. By the Morita
equivalence $\G\sim \tilde{\G}$, in particular Lemma \ref{lem:traces}, we can then replace the positive cone (as an invariant)
with the pairing $R\colon M(Y)\times K_0(C^*(\tilde{\G}))\to \R$, which takes an invariant transverse measure $\mu$ and a $K$-theory element $x$ and produces a real number $\tilde{\tau}_\mu(x)$. By virtue of the mapping torus
structure of $Y$, there is a one-to-one correspondence between $M(Y)$ and the set of $\Z^d$-invariant measures on $X$.

We will compute the invariant $R$ in two stages.
\subsection{Step 1: Identification with pairing via bundles}
We are going to use the leafwise Dirac operator whose $\KKK$-class we have defined in \eqref{eq DX}. See also \cite[Section 1.3]{benoy:gap} (referred to as ``longitudinal Dirac operator'').

First, we identify $R$ with another pairing $R^\prime$, defined as follows. 
Given a tangentially smooth
complex vector bundle $E$ on $Y$, we twist the leafwise Dirac operator $D$ by $E$ and apply the transform \eqref{eq transform} to obtain a tangentially elliptic operator $P_E\in \Psi^0(\tilde{\G},S^{\pm}\otimes E,S^{\mp}\otimes E)$.

Suppose first that $d$ is even. Define
\begin{align*}
	R^\prime\colon M(Y) \times K^0(Y) &\to \R\\
	(\mu,[E]-[F])&\mapsto \Ind_\mu(P_E)-\Ind_\mu(P_F).
\end{align*}

When $d$ is odd we can give the definition by considering
the space $Y\times \R$ foliated by $\mathrm{Leaf}(Y)\times \R$. We then have $K^d(Y)\cong K^{d-1}(Y\times \R)$, given by $[u]\mapsto [E_u]-[E^\prime_u]$, where $u$ is a
unitary and $E_u, E^\prime_u$ are vector bundles. The leafwise Dirac operator $\tilde{D} = D \boxtimes
\slashed{\partial}_\R$ is well-defined and can be twisted by $E_u$ and $E^\prime_u$. Then $R^\prime(\mu,[u])$ is defined as
$R^\prime(\mu,[E_u])=\Ind_\mu(P_{E_u})-\Ind_\mu(P_{E^\prime_u})$.
\begin{prop}\label{prop:ctdirbc}
	Up to the identification given by Connes' Thom isomorphism (indicated by $t_*$ below), the invariant $R$ equals $R^\prime$. In other words, the following
	diagram commutes:
\begin{equation*}
\begin{tikzcd}[remember picture]
	M(Y) \arrow[-,double line with arrow={-,-}]{d} & K^d(Y) \arrow[d, "t_*"] \arrow[r, "R^\prime"] & \R \arrow[-,double line with arrow={-,-}]{d} \\
	M(Y) & K_0(C^*(\tilde{\G})) \arrow[r, "R"] & \,\,\R.
\end{tikzcd}
\begin{tikzpicture}[overlay,remember picture]
	\path (\tikzcdmatrixname-1-1) to node[midway,sloped]{$\times$}
	(\tikzcdmatrixname-1-2);
	\path (\tikzcdmatrixname-2-1) to node[midway,sloped]{$\quad\times$}
	(\tikzcdmatrixname-2-2);
\end{tikzpicture}
\end{equation*}
\end{prop}
\begin{proof}
	By definition of the analytic $\mu$-index, it is enough to show that for any vector bundle $E$ over $Y$, we have
	$t_*[E]=[E]\hot t = \Ind(P_E)\in K_0(C^*(\tilde{\G}))$. This is carried out in detail in \cite[Section 1.3]{benoy:gap}. 
Alternatively, comparing Eq. \eqref{eq t} with the definition of the element $t$ given in \cite{skandalis:fack}, we can deduce that $t$ is the Dirac  morphism underlying the Dirac dual-Dirac approach to the Baum--Connes conjecture with coefficients for $(\R^d,C_0(Y))$ (see also the examples section in \cite{val:shi}). Then the claim follows from the standard index-theoretic interpretation of the assembly map. 
\end{proof}

\subsection{Step 2: Identification with pairing via cohomology}\label{subsec:steptwo}
The second step to compute $R$ is to identify $R^\prime$ with another pairing 
$$C\colon M(Y)\times K^d(Y)\to \R,$$
which we now define. This requires the use of \emph{tangential} de Rham cohomology, for which details can be found in \cite[Chapter 3]{mooscho:globanalysis}. Let us remind the reader of the main points. 

By means of a leafwise differential, one constructs an analogue of the de Rham complex for manifolds, along with associated cohomology
groups $H^*_\tau(Y,\R)$. As in the case of manifolds, there is a useful identification with sheaf cohomology, obtained by considering the sheaf $\R_\tau$ of germs of real-valued continuous functions on $Y$ that are locally constant in the tangential direction. Let $r\colon \R \to \R_\tau$ be the inclusion of the sheaf of real-valued functions that are locally constant. There is an induced map $r_*\colon H^*(Y,\R)\to H^*(Y,\R_\tau)\cong H^*_\tau(Y,\R)$, which is
in general neither injective nor surjective. The \emph{tangential} Chern character can then be defined by composing $r_*$ with the
ordinary Chern character (with values in Čech cohomology, which we could take as a model for $H^*(Y,\R)$). Alternatively, one can
proceed by using tangentially smooth bundles and a leafwise Chern--Weil construction. We will denote the tangential Chern character by
\begin{equation}
\label{eq Chern tau}
\Chern_\tau\coloneqq r_*\circ\, \Chern\colon K^d(Y)\to H^{[d]}_\tau(Y,\R),
\end{equation}
where $H^{[d]}_\tau\coloneqq\oplus_{i\equiv d \text{ mod $2$}}
H^i_\tau$. As in the case of the ordinary Chern character, $\Chern_\tau$ is a natural transformation of $\Z_2$-graded functors. We refer the reader to \cite[Chapter 5]{mooscho:globanalysis} for more details on tangential characteristic classes.

In order to define the pairing $C$, the first ingredient we shall need is the notion of tangential currents. More precisely, having equipped the tangential de Rham complex with an appropriate topology, a current of degree $k$ is by definition an element in the
continuous dual of the space of $k$-differential forms. Such currents form a chain complex whose homology theory we denote by
$H^\tau_*(Y,\R)$. For each $k$, we have an isomorphism $H^k_\tau(Y,\R)^*\cong H^\tau_k(Y,\R)$ \cite[Proposition
3.35]{mooscho:globanalysis}, a fact due essentially to the exactness of the $\mathrm{Hom}$-functor. 

\begin{rema}
By convention, when we view a cohomology group as a topological vector space, we will mean its associated largest Hausdorff quotient. As long as we keep in mind that exactness may fail, this convention does not pose a problem when pairing with homology, as the pairing is well-defined at the level of this quotient.
\end{rema}

Next, we need the \emph{Ruelle--Sullivan map}, which assigns to an invariant transverse measure $\mu$ a top-degree transverse current $C_\mu$. Given a differential form $\omega$ of degree $d$, we define the pairing $C_\mu(\omega)$ as follows. By fixing an orientation on the foliation bundle $FY$, we can think of $\omega$ as a tangentially smooth $1$-density, hence it defines a \emph{tangential measure}, i.e. a measurable section $\ell \mapsto \lambda_\omega^{(\ell)}$ from the leaf space $Y/\tilde{\G}$ to the space of Radon measures
supported in a leaf. 

Now given any tangential measure $\lambda$ and invariant transverse measure $\mu$, we can define a (signed) ``global'' measure on $Y$, written as
\begin{equation*}
	\tilde{\mu}(E)=\int_E \lambda\,d\mu
\end{equation*}
for any Borel subset $E\subseteq Y$. The value of $\tilde{\mu}(E)$ is computed as follows. Let $\ell(y)$ denote the leaf containing $y\in Y$. Fix a complete Borel transversal
$N\subseteq Y$ and a Borel function $f\colon Y\to N$ satisfying $\ell(y)=\ell(f(y))$. For each $x\in N$, define a measure $\rho_x$ on $f^{-1}(x)$ as the restriction of $\lambda^{(\ell(x))}$ to $f^{-1}(x)\subseteq \ell(x)$. If $\chi_E$ denotes the indicator function of
$E$, we set
\begin{equation*}
	\tilde{\mu}(E)=\int_E \lambda\,d\mu\coloneqq\int_N\!\int_{f^{-1}(x)} \chi_E(y)\,d\rho_x(y)d\mu(x).
\end{equation*}
One verifies that this definition does not depend on the choices of $N$ and $f$ \cite[p. 90]{mooscho:globanalysis}. We can now define the value $C_\mu(\omega)$ as the total mass of $\tilde{\mu}$,
\begin{equation*}
	C_\mu(\omega)\coloneqq\int_Y\lambda_\omega\,d\mu.
\end{equation*}

\begin{rema}\label{rem:coibvrema}
This integration procedure can be understood more concretely in the case we are concerned with, where $Y$ is a mapping torus. A tangential differential $k$-form on $Y$ is a $\Z^d$-equivariant continuous map $\omega\colon X\to \Omega_b^k(\R^d)$ into the space of $k$-forms on $\R^d$ with bounded derivatives of all orders. If $\omega$ is a top-degree form, then $C_\mu(\omega)$ is simply the integral
\[
\int_X\int_{(0,1)^d}\omega(x)\,d\mu(x).
\]

Even more concretely, let us consider the case of $1$-dimensional leaves. Let $T$ be a homeomorphism of $X$, and identify $Y$ with the quotient of the cylinder $X\times [0,1]$ by the relation $(x,1)\sim(T(x),0)$. Invariant transverse measures on $Y$ correspond to $T$-invariant measures on $X$. The generic tangential $1$-form is of the form $a(x,t)dt$, with $a(x,1)=a(T(x),0)$, and it is a cocycle, since it is of top-degree. Writing $a=(\partial b/\partial t)(x,t)$, we see that
\[
\int_0^1 a(x,t)\,dt= b(T(x),0)-b(x,0),
\]
hence any tangential $1$-coboundary takes the form $(b_0(T(x))-b_0(x))dt$. Thus we have identified $H^1_\tau(Y,\R)$ with the (topological) $T$-coinvariants of $C(X,\R)$ (cf. \cite[Theorem 2]{benoy:gap}). The dual space is then given by the $T$-invariant measures, as predicted by the Riesz representation theorem.
\end{rema}

The following theorem (whose proof can be found in \cite[Theorem 4.27]{mooscho:globanalysis}) expresses the main property of the
assignment $\mu\mapsto C_\mu$ we have just described.

\begin{theo}\label{thm:riesz}
	Let $Y$ be a compact oriented foliated space with leaf dimension $d$. The Ruelle--Sullivan map induces an isomorphism of vector spaces $M(Y)\cong H^\tau_d(Y,\R)$.
\end{theo}
\begin{rema}
\label{remark orientation}
As explained above, the definition of the Ruelle--Sullivan map involves a choice of orientation of the foliation bundle $FY$. This will play a more active role when we discuss positive cones of $K$-theory.
\end{rema}

We can now define the pairing $C\colon M(Y)\times K^d(Y)\to \R$ by 
\begin{equation}
\label{eq C}
	C(\mu,[E]) = C_\mu(\Chern_\tau([E])),
\end{equation}
where the top-dimensional component of the Chern character is considered. 

In order to identify $R^\prime$ with $C$, we will use the index theorem for measured foliated spaces, which states an equality between the analytic index and the topological index of a differential elliptic operator. As such, let us review the definition of the \emph{topological $\mu$-index} of
a tangentially elliptic differential operator.

We will work with Segal's picture of topological $K$-theory \cite{segal:equivkth}, which is based on elliptic complexes of
vector bundles. The main advantage of this approach is that it allows us to easily bypass the restriction $E\cong E^\prime$ mentioned in Remark \ref{rem:diffbundle}. The even $K$-theory group $K^0(F^*Y)$ of the foliation cotangent bundle is such that each of
its elements can be represented as a complex
\begin{equation*}
	0 \longrightarrow \pi^*(E)\overset{\alpha}{\longrightarrow} \pi^*(E^\prime) \longrightarrow 0,
\end{equation*}
where $E$ and $E^\prime$ are Hermitian vector bundles over $Y$ being pulled back to $F^*Y$. Choosing $\alpha = \sigma(P)$ for a
tangentially elliptic differential operator yields a valid $K$-theory element, which is denoted by $[\sigma(P)]\in K^0(F^*Y)$.

The map $r_*$ introduced earlier (above Eq. \eqref{eq Chern tau}) respects Thom classes. Thus we obtain, for a tangentially smooth complex vector bundle $E\to Y$, the
analogue of the ordinary Thom isomorphism in cohomology $\Phi_\tau\colon H^k_\tau(Y,\R)\to H^{k+n}_{\tau c}(E,\R)$ (with values
in compactly supported cohomology of degree $k+\mathrm{rk}(E)$). The inverse map $\Phi_\tau^{-1}=\pi_!$ is essentially given by
integrating along the fibres of $E$, and as usual we can adjust for the commutativity defect between $\Chern_\tau$ and the Thom
class by introducing a suitable Todd class $\mathrm{Td}_\tau(-)$, defined to be the image under $r_*$ of the ordinary Todd class.

\begin{defi}
\label{def top index}
	The topological $\mu$-index of $P$ is 
	\begin{equation*}
		\Ind_\mu^{\text{top}}(P)= C_\mu(\epsilon \, \pi_!\Chern_\tau([\sigma(P)])\mathrm{Td}_\tau(FY_\C)),
	\end{equation*}
	where $\epsilon=(-1)^{d(d+1)/2}$ is a sign determined by the relative orientation of the tangent bundle and the base space,
	and $FY_\C$ is the complexified foliated tangent bundle.
\end{defi}

We are now ready to state the index theorem for foliated spaces \cite[Theorem 8.1]{mooscho:globanalysis}.

\begin{theo}\label{thm:indthm}
	Let $Y$ be a compact foliated space with leaves of dimension $d$ and oriented foliation bundle $FY$. Let $P$ be the bounded
	transform of a tangentially elliptic (pseudo-)differential operator $D$ and $\mu$ be an invariant transverse measure. Then we
	have an equality of indices
	\begin{equation*}
		\Ind_\mu(P)=\Ind_\mu^{\text{top}}(D).
	\end{equation*}
\end{theo}

We can now prove the identification $R^\prime = C$ that we are after.
\begin{coro}\label{coro:rprimec}
	 We have $R^\prime = C \colon M(Y)\times K^d(Y) \to \R$.
\end{coro}
\begin{proof}
	Assume first that $d$ is even. If $E$ is a vector bundle over $X$, then it is enough to show that the cohomology class associated
	to $D_E$ (appearing in the topological index formula) is equal to $\Chern_\tau(E)$. Since the leaves of $Y$ are
	flat in our case, the Todd class (as well as its square root) is trivial. Now the cohomological equality
	\begin{equation*}
		\Chern(E)= \epsilon\, \pi_!\Chern([\sigma(D_E)]),
	\end{equation*}
	which is well-known from the index theory of Dirac operators \cite[Theorem 13.10]{lawmic:spin}, has an obvious analogue in the tangential setting that follows from the properties of the map $r_*$. In the case that $d$ is odd, the pairing $R^\prime$ is defined by
	considering the obvious foliation on $Y\times \R$. One then checks that
	$\Chern_\tau([u])=\int_\R\Chern_\tau([E_u])\,d\lambda^1$, which is in fact one way to define the odd tangential Chern character (see for instance \cite{mooscho:globanalysis}).
\end{proof}
\subsection{Isomorphism of positive cones}
We are now ready to prove the main result of this section. We make the following conventions. Suppose $f,g\colon
Y_\alpha \to Y_\beta$ are continuous leaf-preserving maps of foliated spaces. A \emph{leafwise} homotopy between $f$ and $g$ is a
leaf-preserving continuous map $h\colon Y_\alpha \times \R \to Y_\beta $ with
\begin{equation*}
	h(x,t)= \begin{cases}
		f(x) & \text{for $t\leq 0$,}\\
		g(x) & \text{for $t\geq 1$,}
	\end{cases}
\end{equation*}
(here $\mathrm{Leaf}(Y_\alpha\times \R)=\mathrm{Leaf}(Y_\alpha)\times \R$).
Leafwise homotopy is clearly an equivalence relation, and we will use the standard notion of \emph{homotopy equivalence} for
continuous leaf preserving maps according to the notion of homotopy just described.

Let $C_\alpha$ (resp. $C_\beta$) denote the pairing $C$ applied to the space $Y_\alpha$ (resp. $Y_\beta$). Recall the $C^*$-algebras $A_\alpha=C^*(\G_\alpha)$ and $A_\beta=C^*(\G_\beta)$.
\begin{theo}\label{thm:trposcon}
	Suppose $F\colon Y_\alpha \to Y_\beta$ is an orientation-preserving leafwise homotopy equivalence of the mapping tori. Then
	we have a commutative diagram
\begin{equation*}
\begin{tikzcd}[remember picture]
	M(Y_\alpha) \arrow[d, "F_*"] & K^d(Y_\alpha)  \arrow[r, "C_\alpha"] & \R \arrow[-,double line with arrow={-,-}]{d}\\
	M(Y_\beta) & K^d(Y_\beta) \arrow[u, "F^*"]\arrow[r, "C_\beta"] & \,\,\R,
\end{tikzcd}
\begin{tikzpicture}[overlay,remember picture]
	\path (\tikzcdmatrixname-1-1) to node[midway,sloped]{$\times$}
	(\tikzcdmatrixname-1-2);
	\path (\tikzcdmatrixname-2-1) to node[midway,sloped]{$\,\times$}
	(\tikzcdmatrixname-2-2);
\end{tikzpicture}
\end{equation*}
where the indicated vertical maps are isomorphisms. Consequently, we have an isomorphism of positive cones 
	\begin{equation*}
		\Phi_\alpha^{-1}F^*\Phi_\beta\colon K_0(A_\beta)^+\longrightarrow
	K_0(A_\alpha)^+,
	\end{equation*}
	where $\Phi_\alpha$ and $\Phi_\beta$ are defined as in \eqref{eq:phidef}.
	\end{theo}
\begin{proof}
	First, recall that $F_*$ is defined using the isomorphism $M(Y)\cong H^\tau_d(Y,\R)$, via the covariant
	functoriality of homology. For a current $c$ on $Y_\alpha$, the pushforward current $F_*c$ is defined by
	$\omega\mapsto c(F^*\omega)$. Implicit in the definition of the Ruelle-Sullivan map is a choice of orientation on $Y$; one sees that if $F$ is orientation-preserving, then $F_*$ preserves positivity of measures. In this way, we obtain a bijection of traces on $A_\alpha$ and $A_\beta$. That the vertical maps are isomorphisms follows from homotopy invariance of (co)homology
	\cite[Corollary 3.19, Proposition 3.35]{mooscho:globanalysis}.
	Let $x\in K_0(A_\beta)^+$, and set
	$\tilde{x}=\Phi_\beta(x)$, $y=\Phi_\alpha^{-1}F^*(\tilde{x})$. By symmetry, it suffices to show that $y$ belongs to
	$K_0(A_\alpha)^+$. This means that for any (positive) $\eta\in M(Y_\alpha)$, we need to show that $\tau_\eta(y)>0$. Set
$\mu=F_*\eta$. The situation can be summarised by the following commutative diagram:
\begin{equation}
\label{eq main diag}
	\,\,\,\,\,\,\,\,\,\begin{tikzcd}[column sep = huge]
		\R \arrow[-,double line with arrow={-,-}]{d} & K_0(C^*(\G_\alpha)) \arrow[l, "\tau_\eta"'] \arrow[d, "m_\alpha"'] &
		K_0(C^*({\G}_\beta)) \arrow[l, "\Phi_\alpha^{-1}F^*\Phi_\beta"'] \arrow[r, "\tau_\mu"] \arrow[d, "m_\beta"] & \R
		\arrow[-,double
		line with arrow={-,-}]{d} \\
		\R \arrow[-,double line with arrow={-,-}]{dd} & K_0(C^*(\tilde{\G}_\alpha)) \arrow[l, "\tilde{\tau}_\eta"'] \arrow[d,
		"t^{-1}_\alpha"']
		& K_0(C^*(\tilde{\G}_\beta)) \arrow[r, "\tilde{\tau}_\mu"] \arrow[d, "t^{-1}_\beta"] & \R \arrow[-,double
		line with arrow={-,-}]{dd} \\
		 & K^d(Y_\alpha) \arrow[d, "\Chern_\tau"']  & K^d(Y_\beta) \arrow[l, "F^*"'] \arrow[d, "\Chern_\tau"] & \\
		\R                                & H_\tau^{[d]}(Y_\alpha,\R) \arrow[l, "C_\eta"']           & H_\tau^{[d]}(Y_\beta,\R) \arrow[l,
		"F^*"'] \arrow[r, "C_{\mu}"]           & \,\,\R.
	\end{tikzcd}
	\vspace*{2ex}
\end{equation}
Commutativity of the left and right sides comes from Lemma \ref{lem:traces} and the index theorem for foliated spaces, Theorem \ref{thm:indthm}, together with the discussion preceding Definition \ref{def top index}. Commutativity of the top and bottom central squares follows from the definition of $\Phi$ and naturality of $\Chern_\tau$.
	By the commutativity of the left side, we have
$\tau_\eta(y)=C_\eta(\Chern_\tau(F^*\tilde{x}))$. Naturality of the Chern character implies that
\begin{equation*}\tau_\eta(y)=C_\eta(F^*(\Chern_\tau(\tilde{x})))=C_\mu(\Chern_\tau(\tilde{x})).
\end{equation*}
	By the commutativity of the right side,
	$C_\mu(\Chern_\tau(\tilde{x}))$ is equal to $\tau_\mu(x)$, which is a positive number. This completes the proof.
\end{proof}

\subsection{The twisted case}\label{subsec:twist}

The case of twisted crossed products is handled by modifying the topological side of the index formula in Theorem \ref{thm:indthm}. The details of this change are worked out in \cite[Section 3]{benmat:magzero}. As before, let us denote by $\Theta$ the $(d\times d)$ skew-symmetric matrix (with real entries) determining the twisting cocycle. 

Let $\eta$ be the $1$-form satisfying $B=d\eta$ for the closed $2$-form defined by $\Theta$ as explained below \eqref{eq:covariance}. Following the pattern in Proposition \ref{prop:ctdirbc}, we now consider the Dirac operator on $\R^d$ and perform the following twisting procedure to arrive at a suitable leafwise Dirac operator. Let $\nabla=d+i\eta$ be the connection on the trivial line bundle on $\R^d$, and $\nabla_E$ the lift to $X\times \R^d$ of a connection on a complex vector bundle $E\to Y$. Consider the twisted leafwise Dirac operator
\[
D_E=D\otimes \nabla\otimes \nabla_E\colon L^2(X\times \R^d,S^+\otimes E)\to L^2(X\times \R^d,S^-\otimes E).
\]
The analogue of the Wassermann idempotent (see \cite[above Lemma 2.5]{conmosc:nov} and \cite[Section 3.3]{benmat:magzero}) yields an element $e(D_E)\in M_2(C^*(\tilde{\G},\Theta))\otimes\cK$. Defining $P_E$ to be the degree-zero transform of $D_E$, we define the analytic twisted $\mu$-index to be 
\[
\Ind_\mu(P_E)=\tilde{\tau}_\mu([e(D_E)]-[\mathrm{diag}(0,1)]),
\]
and the associated pairing $R^\prime$ sending $\mu$ and $[E]$ to $\Ind_\mu(P_E)$.

After these adjustments, let us turn to the topological pairing $C$ from \eqref{eq C}. We focus on the case of $d$ even, since the case of $d$ odd can be reduced to this by following the considerations made in the paragraph above Proposition \ref{prop:ctdirbc}. First of all, recall that the determinant of a skew-symmetric matrix can always be written as the square of a polynomial in the matrix entries. This is a polynomial with integer coefficients (dependent only on the size of the matrix) whose value, when applied to the coefficients of a skew-symmetric matrix, is called the \emph{Pfaffian} of that matrix. The Pfaffian is nonvanishing only for $(2n \times 2n)$ skew-symmetric matrices, in which case it is a polynomial of degree $n$. 

If $I$ is a subset of $\{1,\dots,d\}$ of even cardinality, denote by $\Theta_I$ the matrix obtained from $\Theta$ by considering only rows and columns indexed in $I$, and write $\mathrm{Pf}(\Theta_I)$ for the resulting Pfaffian value. Denote also by $dx_I$ the differential form of degree $\lvert I \rvert$ on the torus $\mathbb{T}^d$, after it has been lifted to $Y$. The pairing $C$ is defined by sending $(\mu,[E])$ to the following sum:
\[
\sum_I \mathrm{Pf}(\Theta_I) \cdot C_\mu(dx_I\wedge \Chern_\tau([E])).
\]

The measured \emph{twisted} foliated index theorem, combined with an argument analogous to Corollary \ref{coro:rprimec} for simplifying the topological index formula, then gives us the identification $R^\prime= C$ that we are after. See \cite[Theorem 3.2]{benmat:magzero} for the proof of the measured twisted foliated index theorem. To summarise, we have:	

\begin{theo}\label{thm:cctwist}
We have $R^\prime = C \colon M(Y)\times K^d(Y) \to \R$.
\end{theo}

\section{Main results and applications}
\label{sec:results}
In this section we assemble the results from the previous sections and draw the main conclusions of the paper. We also present some applications of the techniques we have developed. 

\subsection{Geometric construction of the Elliott invariant}
\label{subsec:geometric Elliott}
Recall the $K$-theory map $\iota_*\colon K^d(Y^0)\to K^d(Y)$ induced by the open inclusion of the fundamental domain into the mapping torus, namely $\iota\colon X\times (0,1)^d\hookrightarrow Y$. Let $T$ be the natural pairing $H^\tau_d(Y,\R)\times K^d(Y)\to \R$, defined through the pairing $C$ from \eqref{eq C}, combined with Theorem \ref{thm:riesz}.

We define the geometric Elliott invariant as the tuple
\[
\mathrm{GEll}(A^\Theta) = (K^0(Y),\iota_*,K^1(Y),H^\tau_d(Y,\R),T).
\]

\begin{rema}
It is worth pointing out that while $\mathrm{GEll}$ is designed to be a ``classifier'' of $C^*$-algebras, its definition involves the underlying geometric structure from which the $C^*$-algebra in question arises. In particular, having included de Rham homology (rather than normalised positive currents only) in its definition, $\mathrm{GEll}$ does not correspond precisely to the standard Elliott invariant. Instead, it should be thought of as an invariant of a $C^*$-algebra for a given geometric setting. 

In a related spirit: if we assume $X$ has a smooth structure, then there exist examples where $A_\alpha\cong A_\beta$, but $C^\infty(X)\rtimes_\alpha \Z^d$ is not isomorphic to $C^\infty(X)\rtimes_\beta \Z^d$ \cite{hl:smoothsph,hl:smoothman}. In view of this, it could be interesting in the future to introduce an element of cyclic (co)homology to supplement $\mathrm{GEll}$.
\end{rema}

An isomorphism of such invariants consists of isomorphisms between the corresponding $K$-theory and homology groups, making the triangle \eqref{eq basic diagram} commute and preserving the pairing $T$, along with the Choquet simplex structure of positive currents.

\begin{theo}\label{thm:classres}
If $\alpha,\beta \colon \Z^d\to \mathrm{Homeo}(X)$ are two free and minimal actions, then $\mathrm{GEll}(A_\alpha^\Theta)\cong \mathrm{GEll}(A^\Theta_\beta)$ implies $A_\alpha^\Theta\cong A^\Theta_\beta$ as $C^*$-algebras.
\end{theo}

\begin{proof}
As pointed out previously, $A^\Theta$ belongs to the UCT-class because its associated groupoid is amenable \cite[Lemma 10.6]{Tu99} (this also follows from the strong Baum--Connes conjecture \cite{nestmeyer:loc, val:kthpgrp}). Amenability also implies the associated algebra is nuclear \cite{lal:nuc}. In Section \ref{sec:prelims} we have explained that our assumption on $X$ having finite covering dimension leads to $\mathcal{Z}$-stability.
Now the claim that $\mathrm{GEll}(A_\alpha^\Theta)\cong \mathrm{GEll}(A^\Theta_\beta)$ implies $A_\alpha^\Theta\cong A^\Theta_\beta$ follows from the classification statement for the usual Elliott invariant (see Theorem \ref{thm:classification} and the references given there), coupled with Proposition \ref{prop:tcob}, Lemma \ref{lem:twistunit}, Proposition \ref{prop:presordunit}, Corollary \ref{coro:rprimec}, and Theorem \ref{thm:cctwist}.
\end{proof}

If the $K$-theory of $A^\Theta$ separates traces (e.g., if $A^\Theta$ is a crossed product of the Cantor set), which is the main case of interest for this paper, we need only consider (tangential) de Rham homology as a vector space. This data is enough to recover ordered $K$-theory. This invariant is enough for classification thanks to the main result in \cite{tww:quasidiag}, \cite[Theorem B]{cetww:nucdim} (which implies $A^\Theta$ has decomposition rank at most $1$), and \cite[Theorem 1.3]{win:classprod} (based on \cite{lin:ctoprkzero,lin:asy,linniu:lift,winter:loc}).

\subsection{Noncommutative rigidity of mapping tori}
\label{subsec:rigidity}
We now prove the main theorem on noncommutative rigidity for mapping tori. Broadly speaking, the classification of $C^*$-algebras can be described as the problem of lifting (invertible) maps of Elliott invariants to $\ast$-isomorphisms of $C^*$-algebras. On the other hand, in geometry and topology, it is customary to investigate whether or not a certain homotopy equivalence of spaces can be lifted to a homeomorphism. For foliated spaces whose classical leaf space is singular, one can approach the problem of topological rigidity from a $C^*$-algebraic perspective,
 where instead of an isomorphism of foliated spaces, one seeks an isomorphism 
of ``noncommutative foliated spaces''.

In the case of the foliated space $Y$ we have been studying, the associated ``noncommutative space'' is the foliation $C^*$-algebra $C^*(\tilde{\G})$, or $C^*(\G)$ if we allow for Morita equivalence. In this context, we have the following rigidity result:

\begin{theo}\label{thm:lift}
Any basic leafwise homotopy equivalence $Y_\alpha\cong Y_\beta$ can be lifted to $*$-isomorphisms
\[
C^*(\G_\alpha) \cong C^*(\G_\beta), \qquad C^*(\tilde{\G}_\alpha)\otimes\cK\cong C^*(\tilde{\G}_\beta)\otimes\cK
\]
of the associated foliation $C^*$-algebras.
\end{theo}
\begin{rema}
The isomorphism in Theorem \ref{thm:lift} is a ``lift'' of the given homotopy equivalence in the sense that both induce the same isomorphism $\mathrm{GEll}(C^*(\G_\alpha))\cong \mathrm{GEll}(C^*(\G_\beta))$.
\end{rema}

\begin{proof}
Let $F\colon Y_\alpha\to Y_\beta$ denote the homotopy equivalence. First we show that $F_*\colon  H^\tau_d(Y_\alpha,\R)\to H^\tau_d(Y_\beta,\R)$ induces an isomorphism of the trace simplices of the crossed products. Recall for a current $c$, the element $F_*c$ is defined as $F_*c(\omega)=c(F^*\omega)$ using the pullback of (tangential) differential forms. The trace simplex is identified with the compact convex subspace of invariant (probability) measures on $X$, viewed in the weak-$\ast$ topology of linear functionals on $C(X)$. Since the measures are invariant, it is clear this topology is determined by pointwise convergence on the coinvariants $C(X)_\alpha$, i.e. $\mu_n\to \mu$ if and only if $\mu_n([f])\to \mu([f])$ for any class in $[f]\in C(X)_\alpha$. By \cite[Theorem 2]{benoy:gap} (see also \cite[pp.~ 106-108]{mooscho:globanalysis}), there is an isomorphism $\int_\ell\colon H_\tau^d(Y,\R)\to C(X)_\alpha$ given by integration along the leaves, hence the relevant topology on de Rham homology is determined as $c_n\to c$ if and only if $c_n(\omega)\to c(\omega)$ for any $\omega\in H^d_\tau(Y,\R)$ (recall we have $c=C_\mu$ for some $\mu$ thanks to the Riesz representation theorem). Then $F_*$ is a continuous linear map in this topology, inducing a continuous bijection of transverse measures $F_*\colon M(Y_\alpha)\to M(Y_\beta)$ (cf. \cite[Proposition 4.29]{mooscho:globanalysis}). Let us check that its restriction to those currents corresponding to normalised measures (hence tracial states) is well-defined. Suppose that $\int_\ell\omega_\beta=1$ identically as a function of $Y_\beta$. Then $\omega$ is in the same tangential cohomology class as the form $\eta_\beta(x,\textbf{t})=dt^1\wedge\cdots\wedge dt^d$. Now observe that $\eta_\beta$ is (the top-degree part of) the image of $1_{C(X)}$ along the vertical map $\Chern_\tau \circ\, t^{-1}_\beta \circ m_\beta$ in diagram \eqref{eq main diag}. By naturality of $\Chern_\tau$ and the fact that $F$ is basic, $F^*\omega_\beta$ is cohomologous to $\eta_\alpha$, hence $\int_{Y}F^*\omega_\beta\,d\mu=\int_{Y}\eta_\alpha\,d\mu=1$ for any transverse probability measure $\mu$. Because of the homotopy invariance of cohomology, $F_*$ induces a continuous bijection of the trace simplices, hence a homeomorphism since they are compact and Hausdorff. Now the result follows from a combination of Theorem \ref{thm:trposcon} and Theorem \ref{thm:classres}.
\end{proof}

\subsection{Generalisations}\label{subsec:gen}

We have so far considered the case of $\Z^d$ as a discrete subgroup of $\R^d$. As we now show, most of our constructions and arguments carry over to the broader setting of discrete cocompact subgroups of simply connected solvable Lie groups.

Let $H$ be a discrete cocompact subgroup of a simply connected solvable Lie group $G$. It is well-known that $G$ is diffeomorphic to $\R^d$ \cite[Theorem 3.18.11]{var:lie}. Since $H$ acts freely and properly on $G$ (by left translation), we can take $EH=G$, i.e. $G$ is a model for the classifying space of $H$-bundles. Let $X$ be a space equipped with a free and minimal $H$-action, with homotopy quotient $Y=X\times_H G$.

Since solvable groups (along with their closed subgroups) are amenable \cite[Chapter 4]{tulliocoor:auto}, we have that $A=C(X)\rtimes H\cong C(X)\rtimes_r H$ is nuclear \cite{lal:nuc}, simple, unital, and belongs to the UCT-class (see \cite[Lemma 10.6]{Tu99} or \cite[Proposition 9.5]{nestmeyer:loc}). The groupoids $\G=H\ltimes X$ and $\tilde{\G}=G\ltimes Y$ are Morita equivalent by the same ``reduction-to-transversal'' argument that underlies Lemma \ref{lem:morita} (see \cite[Examples 2.6 \& 2.7]{murewi:morita} and \cite[Section 3]{put:spiel} for details). Interpreting $Y$ as the induction space $Y=\Ind_H^G(X)$, the analytic counterpart of this Morita equivalence is the stable isomorphism
\[
C(X)\rtimes H \otimes \cK \cong C_0(\Ind_H^G(X))\rtimes G \otimes \cK
\]
obtained via the standard Mackey-Green-Rieffel machine \cite[Section 3 \& Appendix~B]{ekqr:cat}.

As $G$ is contractible, $\tilde{\G}$ is the holonomy groupoid associated to the foliation $Y$ \cite[p.~44]{mooscho:globanalysis}. The leaves in $Y$ correspond to the $G$-orbits. Since $Y$ is a fibre bundle over the compact space $G/H$ with compact fibres (hence compact), the foliated index theorem holds for $(X,H,G)$ just as it did for $(X,\Z^d,\R^d)$.

Since $G$ is diffeomorphic to $\R^d$, we also have Bott periodicity. This allows us to characterise the ``inclusion of the fundamental domain'' in a way analogous to
Lemma \ref{lem:fac}. In other words, we can map $[1]\in K_0(C(X))$ via the composition of maps 
\begin{equation}
\iota^G\colon K_0(C(X))\cong K_d(C_0(X\times G)) \to K_d(C_0(X\times G)\rtimes H) \cong K_d(C(Y)),
\end{equation}
where the first is defined by Bott periodicity, the second is induced by inclusion at $1\in H$, and the third is Rieffel's Morita equivalence for the free and proper $H$-action on $X\times G$. In this more general setting, on the basis of Definition \ref{def:basic}, we define a map $F\colon Y_\alpha\to Y_\beta$ to be basic by considering the following (not necessarily commutative) triangle,
\begin{equation*}
	\xymatrix{K_0(C(X)) \ar[r]^-{\iota^G_{\beta}} \ar[dr]_-{\iota^G_{\alpha}} & K_d(C(Y_\beta)) \ar[d]^-{F^*} \\
	& K_d(C(Y_\alpha)),}
\end{equation*}
and requiring $F^* \iota^G_\beta[1] = \iota^G_\alpha[1]$.

Since $G$ is a simply connected solvable Lie group, any crossed product involving $G$ is isomorphic to an iterated crossed by $\R$ \cite[Examples 10.1.1~(c)]{black:kth}. This is a generalisation of the fact that group extensions give rise to twisted crossed product $C^*$-algebras of the form $C^*(N)\rtimes_\sigma Q$ (and if $Q=\R$ then we can ``untwist''). This means that we have a corresponding version of Connes' Thom isomorphism: for any separable $G$-$C^*$-algebra $B$, there is a $\KKK$-equivalence between $B\rtimes G$ and $S^d B$ \cite{skandalis:fack}.

With this in mind, we have an isomorphism 
$$K_*(C(X)\rtimes H)\cong K^{*+d}(Y).$$ 
The main arguments we relied on, which were based on Morita equivalence and the foliation index theorem, also carry over to this more general setting. In particular, we have direct analogues of Proposition \ref{prop:presordunit} and Corollary \ref{coro:rprimec}, and this allows us to construct the geometric Elliott invariant $\mathrm{GEll}$ in this greater generality. 

We may summarise this discussion as follows.

\begin{theo}\label{thm:classgeneral}
Let $H$ and $G$ be as above. Let $\alpha,\beta \colon H\to \mathrm{Homeo}(X)$ be two free and minimal actions such that the corresponding crossed products absorb the Jiang-Su algebra $\mathcal{Z}$. Then $\mathrm{GEll}(A_\alpha)\cong\mathrm{GEll}(A_\beta)$ implies $A_\alpha\cong A_\beta$ as $C^*$-algebras.
\end{theo}

For example, in the case that $(H,G)$ is a pair where $G$ is a nilpotent Lie group and $H$ is a finitely
generated, infinite, nilpotent subgroup, Theorem \ref{thm:classgeneral} applies by virtue of \cite[Corollary 7.6]{gwz:rokhlin} and the proof of $\mathcal{Z}$-stability as a consequence of finite nuclear dimension. A specific example is given by $(H,G)=(H_3(\Z),H_3(\R))$, the pair consisting of the discrete and continuous Heisenberg groups.

On the other hand, there exist many pairs $(H,G)$ where $G$ is simply connected and non-nilpotent. Under the Bianchi classification of solvable
non-nilpotent Lie groups, the group $6_0$ (the isometry group of $\mathrm{Solv}$ geometry as in Thurston's classification of model
geometries) has cocompact lattices that correspond
to compact Riemannian $3$-manifolds whose universal cover is isometric to $\mathrm{Solv}$ geometry \cite[Section 5]{bock:solv}. Similarly, the group $7_0$ (a subgroup of the full isometry group of Euclidean geometry) admits as cocompact lattices the $3$-dimensional crystallographic groups.

\subsection{Further applications}
\label{subsec:app}
We give two more applications of the theory developed in this paper. The first gives a sufficient condition under which the $C^*$-algebra $A$ is classified by a simpler invariant, namely ordered $K$-theory.

Recall the tangential Chern character $\Chern_\tau$ from \eqref{eq Chern tau}. Note that in general, rationalising this map does not lead to an isomorphism.

\begin{prop}
\label{prop K thy sep traces}
Suppose the map 
\begin{equation}\label{eq:rchcar}
\Chern_\tau\otimes 1 \colon K^d(Y)\otimes \R \to H_\tau^{[d]}(Y,\R)
\end{equation}
associated to the tangential Chern character is surjective. Then the $K$-theory group $K_0(A)$ separates traces.
\end{prop}
\begin{proof}
Let $\tau_\mu,\tau_\eta$ be two distinct traces not distinguished by $K_0(A)$. With notation as in diagram \eqref{eq main diag} in the proof of Theorem \ref{thm:trposcon}, the map $(m\circ t^{-1} \circ \Chern_\tau)\otimes \,\mathrm{id}_\R$ is surjective by hypothesis. Hence the index theorem (i.e., the commutativity of both sides of diagram \eqref{eq main diag}) implies that $C_\mu=C_\eta$ as elements of $H^\tau_d(Y,\R)$. Now 
Theorem \ref{thm:riesz} implies that $\mu=\eta$, which is a contradiction.
\end{proof}

Let us make two remarks on the hypothesis in Proposition \ref{prop K thy sep traces}. The first is that in order to obtain the conclusion, we only need the map in \eqref{eq:rchcar} to be a surjection onto top-degree cohomology. Second, since the \emph{ordinary} Chern character $\Chern$ is an isomorphism after tensorization by $\R$, it follows from the discussion in Section \ref{subsec:steptwo}
that the hypothesis is equivalent to surjectivity of the map $r_*\colon H^d(Y,\R)\to H_\tau^d(Y,\R)$ discussed there.

\begin{rema}
When $X$ is a Cantor set, the map $r_*\colon H^d(Y,\R)\to H_\tau^d(Y,\R)$ is identified with the map between coinvariants $C^\infty(X,\R)_\alpha \to C(X,\R)_\alpha$ whose image is dense (see \cite[Corollary 4.2]{ros:gaplab}; here ``smooth'' is understood as locally constant). If in addition there is a unique $\alpha$-invariant probability measure on $X$ (the uniquely ergodic case), then $r_*$ is a surjection (see \cite[Remark 5.4]{ros:gaplab}).
\end{rema}

\begin{theo}\label{thm:class}
$C^*$-algebras of the form $A=C(X)\rtimes \Z^d$ are classified by the ordered $K$-theory invariant
\[
(K_0(A),K_0^+(A),[1_{A}],K_1(A))
\]
whenever the map in Eq. \eqref{eq:rchcar} is surjective.
\end{theo}
\begin{proof}
This follows from \cite[Theorem 1.3]{win:classprod} (which is based on \cite{lin:ctoprkzero,lin:asy,linniu:lift,winter:loc}) together with Proposition \ref{prop K thy sep traces}.
\end{proof}

Now we apply Theorem \ref{thm:classres} to the gap-labelling problem for Cantor minimal systems \cite{bel:hull}, in particular to clarify the scope of applicability of existing results in the literature.

First let us provide some context for the problem and its relation to noncommutative geometry. The term ``gap-labelling'' comes from solid-state physics, in particular the study of spectral gaps of Hamiltonian operators describing the motion of conduction electrons inside a given disordered material (e.g. a quasi-crystal). The ``labels'' are the trace values associated to projections onto spectral values and have a physically meaningful interpretation as values of the \emph{integrated density of states}, which, roughly speaking, measure the number of available energy levels per unit volume. From a physical viewpoint, formulating the problem using traces and $K$-theory is useful because it implies that the gap labels are invariant under norm-resolvent perturbation of the Hamiltonian, which in turn yields \emph{sum} and \emph{conservation rules} as the dynamics are perturbed. Roughly speaking, these rules underlie formulas for transitions between energy levels that regulate absorption and emission of electromagnetic radiation.

This mathematical framework and its physical ramifications are broadly described as ``noncommutative Bloch theory'', a program initiated by Bellissard \cite{bel:gap} and which has the gap-labelling conjecture as a fundamental statement.

Modulo a few simplifications and identifications, a $C^*$-algebra of the type studied in this paper is proposed as the \emph{$C^*$-algebra of observables} for the physical system under consideration. Due to the disordered nature of the material, and its mathematical modeling through tiling systems \cite{kell:put}, the Cantor set appears in the form of the space $X$ in this paper, and we are led to consider the crossed product $C(X)\rtimes \Z^d$ associated to a free and minimal action. The addition of a constant magnetic field leads to the \emph{magnetic} gap-labelling problem. This corresponds to taking a twisted crossed product with respect to a group $2$-cocycle, as discussed in Section \ref{sec:prelims} of this paper. 

Let us focus first on the untwisted case. Fix a $\Z^d$-invariant probability measure $\mu$ on the Cantor set $X$ (often assumed to be ergodic in physical situations). We obtain a dual tracial state $\tau_\mu$ on the crossed product algebra $A=C(X)\rtimes \Z^d$. The problem consists of proving the following equality between subsets of the real numbers:
\begin{equation}\label{eq:gap}
\tau_\mu (K_0(A)) = \mu_*(K_0(C(X))).
\end{equation}

It is routine to compute $K_0(C(X))\cong C(X,\Z)$ (while the odd $K$-group is trivial), so the right-hand side of \eqref{eq:gap} is also denoted by $\Z[\mu]$ -- the group of integer linear combinations of measures of clopen subsets of $X$. Following \cite{benmat:mag}, let us consider the case where the mapping torus associated to $(X,\Z^d)$ is a \emph{principal solenoidal torus}. In \cite{benmat:mag}, the authors were able to prove \eqref{eq:gap} (in the magnetic case) using index theory on foliated spaces with the favourable topology of solenoidal tori. 

A space $Y$ is a principal solenoidal torus if it is an inverse limit of tori $Y_j=\T^d$, where the connecting maps $f_{j+1}\colon Y_{j+1}\to Y_j$ are Galois covers of degree greater than one (for example, $f_j(z_1,\dots,z_d)=(z_1^{p_1},\dots,z_d^{p_d})$, where $p_i\in\N$, $p_i>1$). Recall that $f_j$ is called Galois (or regular or normal) if the group $\Aut(f_j)$ of deck transformations acts transitively on each fibre. 

Note that $\Aut(f_j)$ always acts freely, hence $f_j$ is a principal $G$-bundle where $G=\Aut(f_j)$ is considered as a discrete group. Let $p_j\colon X_j\to X_1=\T^d$ be the canonical map. Then $p_j$ is a (finite) Galois cover, and in particular the group $\Aut(p_j)$ is isomorphic to a quotient of the fundamental group $\pi_1(Y_1)\cong \Z^d$. Let $\Aut(p_j)=\Z^d/\Gamma_j$. The groups $\{G_j=\Z^d/\Gamma_j\}_j$ form an inverse system whose inverse limit we denote by $\Omega$.  Note that $\Omega \to Y \to Y_1$ is a principal $\Omega$-bundle.

One may consider the following problem: \emph{characterise those Cantor minimal systems whose associated homotopy quotient is a principal solenoidal torus}. This question helps clarify the range of applicability of the main result in \cite{benmat:mag}. Since the gap-labelling conjecture is expressed in terms of $C^*$-algebras, the hypothesis on the principal solenoidal torus may appear to be somewhat unnatural, since it involves the homotopy quotient $Y$, which is only indirectly related to the problem. The results of the present paper are useful in precisely this situation, as they provide a way to translate between properties of the homotopy quotient and properties of the crossed product.

\begin{prop}\label{prop:exgap}
Let $(X,\Z^d,\alpha)$ be a Cantor system as above. Suppose $Y_\alpha$ is isomorphic to a principal solenoidal torus (as a fibre bundle over the torus). Then the $C^*$-algebra $A=C(X)\rtimes \Z^d$ is isomorphic to a $\Z^d$-odometer system crossed product.
\end{prop}

Note that Proposition \ref{prop:exgap} does not characterise the dynamical system \emph{per se}, but rather the $C^*$-algebra it generates. It provides a partial answer to \cite[Problem 1]{benmat:mag}.

Given a principal solenoidal torus $X$, there is a natural action $\beta\colon\Z^d \to \mathrm{Homeo}(\Omega)$ given by coset translation on each group $G_j$, and the resulting system $(\Omega,\Z^d,\beta)$ is called a \emph{$\Z^d$-odometer} \cite{gps:zdodom}. The $C^*$-algebra $C(\Omega)\rtimes \Z^d$ belongs to a more general class of $C^*$-algebras known as \emph{generalised Bunce-Deddens algebras} \cite{orf:genbd}, and in particular it is a classifiable $C^*$-algebra of real rank zero, stable rank one, with comparability of projections and a unique trace. 

\begin{proof}
By construction, $Y_\beta$ is isomorphic to $Y_\alpha$. The claim then follows directly from Proposition \ref{prop:either} and Theorem \ref{thm:lift}, noting that separation of traces is automatic for a Cantor set.
\end{proof}

This result could be viewed as clarifying the scope of applicability of the main theorem in \cite{benmat:mag}. Having characterised which crossed products are involved, it is easy to prove that the associated systems of the form $(\Omega,\Z^d,\beta)$ satisfy the gap-labelling conjecture by a direct argument. This is most likely well-known to experts, but we provide a brief argument here for completeness.

\begin{prop}\label{eq:concgap}
Equation \eqref{eq:gap} holds for the odometer system $(\Omega,\Z^d,\beta)$.
\end{prop}
\begin{proof}
Since $\beta$ is defined as the ``limit'' of an action on each group $G_j$ (viewed as a discrete space), the crossed product $C^*$-algebra $C(\Omega)\rtimes \Z^d$ is easily seen to be isomorphic to an inductive limit of algebras $A_j= C(G_j)\rtimes \Z^d \cong C^*(\Z^d/\Gamma_j)\rtimes \Z^d$. We deduce a standard Morita equivalence between $A_j$ and $C^*(\Gamma_j)$, and it is routine to check that this Morita equivalence transports the dual trace on $A_j$ to the canonical tracial state on $C^*(\Gamma_j)\cong C(\mathbb{T}^d)$, scaled by a factor of $1/[\Z^d\colon\Gamma_j]$. It is well-known that this state takes integer values. 
As the Haar measure on $\Omega$ is induced by the unique invariant measures $\mu_j$ on each $G_j$, the equality \eqref{eq:gap} follows.
\end{proof}

Note that the integrality aspect of the previous proof holds as long as each $\Gamma_j$ is a torsion-free group satisfying the Baum--Connes conjecture (see for example \cite{valjens:misb}). Hence the equality \eqref{eq:gap} holds for generalised Bunce-Deddens algebras beyond the odometer type.

The presence of a non-zero magnetic field is accompanied by the introduction of a $\Z^d$-cocycle determined by a skew-symmetric matrix $\Theta$, as discussed in Section \ref{sec:prelims}. Then  Proposition \ref{prop:exgap} still holds, replacing the odometer $C^*$-algebras with a twisted analogue. More precisely, the $C^*$-algebra $A^\Theta$ is isomorphic to an inductive limit of algebras $A_j$, each of which is Morita equivalent to $C^*_\Theta(\Gamma_j)$, i.e. a higher dimensional noncommutative torus.

In order to discuss the twisted analogue of Proposition \ref{eq:concgap}, one needs to modify the right-hand side of \eqref{eq:gap} by introducing the Pfaffian of $\Theta$ (see \cite{benmat:magzero} for more details). This should come as no surprise in view of the twisted foliated index theorem discussed in Subsection \ref{subsec:twist}. The proof of Proposition \ref{prop:exgap} in the twisted setting proceeds by replacing integrality of the standard trace on $C^*(\Gamma_j)$ with the computation of trace values on noncommutative tori in \cite{ell:proj}. 

\medskip

Having mentioned noncommutative tori, we conclude this paper with a simple result concerning a type of topological rigidity of the \emph{Kronecker foliation}. Let $\Z$ act on $X=S^1$ by an irrational rotation of angle $2\pi\theta$, where $\theta\in [0,1]$. The crossed product $C(X)\rtimes \Z$ is then a noncommutative torus, also known as the irrational rotation algebra $A_\theta$. This algebra is Morita equivalent to the Kronecker foliation $C^*$-algebra, namely the holonomy groupoid $C^*$-algebra associated to the foliation on the torus $\mathbb{T}^2$ whose foliation bundle is
\[
F\mathbb{T}^2 = \{ v\in T\mathbb{T}^2 \mid \omega(v) = 0\},
\]
where $\omega=a_1dx-a_2dy$ is a smooth $1$-form with $a_1/a_2 = \tan(2\pi\theta)$. The isomorphism class of $A_\theta$ is well-understood \cite{rieffel:nc}, and in particular $A_\theta\cong A_\zeta$ if and only if $\theta=\zeta$, or $\theta=1-\zeta$. Each leaf sits densely in the torus and corresponds to an ``irrational flow''. 

In this case, the mapping torus $Y$ associated to $(X,\Z)$, viewed as a foliated space, is isomorphic to the $2$-torus $(\mathbb{T}^2,\theta)$ foliated by the irrational flow. Then an application of Theorem \ref{thm:lift} gives the following:

\begin{coro}
There is no basic leafwise homotopy equivalence between $(\mathbb{T}^2,\theta)$ and $(\mathbb{T}^2,\zeta)$ unless $\theta=\zeta$ or $\theta=1-\zeta$.
\end{coro}

\addtocontents{toc}{\SkipTocEntry}\section*{Declarations}
\addtocontents{toc}{\SkipTocEntry}\subsection*{Funding}
The second author acknowledges support by: Foreign Young Talents' grant QN2021137002L (National Natural Science Foundation of China), CREST Grant Number JPMJCR19T2 (Japan Science and Technology Agency of the Ministry of Education, Culture, Sports, Science and Technology), Marie Skłodowska-Curie Individual Fellowship (Horizon Europe, European Commission, Project No.~101063362).
\addtocontents{toc}{\SkipTocEntry}\subsection*{Data availability statement}
This manuscript has no associated data.
\addtocontents{toc}{\SkipTocEntry}\subsection*{Competing interests}
All authors certify that they have no affiliations with or involvement in any organization or entity with any financial interest or non-financial interest in the subject matter or materials discussed in this manuscript.

\backmatter

\bibliographystyle{amsplain-init-nodash}
\bibliography{BibliographyBST}   
\end{document}